\documentclass[11pt, reqno]{amsart}
\usepackage{amssymb,latexsym,amsmath,amsfonts}
\usepackage[mathscr]{eucal}

\numberwithin{equation}{section}

\theoremstyle{definition}
\newtheorem{definition}{Definition}[section]

\newtheorem{remark}[definition]{Remark}

\theoremstyle{plain}
\newtheorem{theorem}[definition]{Theorem}
\newtheorem{lemma}[definition]{Lemma}

\newtheorem{result}[definition]{Result}

\newtheorem{question}{Question}
\newtheorem{conjecture}{Conjecture}

\newcommand{\eps}{\varepsilon}
\newcommand{\zt}{\zeta}
\newcommand{\al}{\alpha}
\newcommand{\lm}{\lambda}

\newcommand{\Tht}{\Theta}

\newcommand{\dl}{\delta}

\newcommand{\sm}{\sigma}
\newcommand{\om}{\omega}

\newcommand\ba[1]{\overline{#1}}
\newcommand\hull[1]{\widehat{#1}}
\newcommand\wtil[1]{\widetilde{#1}}

\newcommand{\rl}{\Re\mathfrak{e}}
\newcommand{\imag}{\Im\mathfrak{m}}

\newcommand{\id}{\mathbb{I}}
\DeclareMathOperator{\Tr}{Tr}

\newcommand{\bdy}{\partial}
\newcommand{\OM}{\Omega}

\newcommand{\ball}{\mathbb{B}}
\newcommand{\disc}{\mathbb{D}}

\newcommand{\smoo}{\mathcal{C}}

\newcommand{\poly}{\mathscr{P}}

\newcommand{\impl}{\Longrightarrow}

\newcommand{\CC}{\mathbb{C}^2}
\newcommand{\cplx}{\mathbb{C}}
\newcommand{\RR}{\mathbb{R}^2}
\newcommand{\rea}{\mathbb{R}}

\begin{document}

\title[Certain smooth real surfaces in $\mathbb{C}^2$ with isolated singularities]{Certain real surfaces  in $\mathbb{C}^2$ with isolated singularities}
\author{Sushil Gorai}
\address{Department of Mathematics and Statistics, Indian Institute of Science Education and Research Kolkata,
Mohanpur, Nadia, West Bengal 741246, India}
\email{sushil.gorai@iiserkol.ac.in}

\keywords{Polynomially convex; totally real; CR singularity, polynomial hull, analytic disc}
\subjclass[2020]{Primary: 32E20, 32D10; Secondary: 30E10}

\begin{abstract}
 The surfaces in $\cplx^2$ with an isolated CR singularity at the origin and with cubic lowest 
degree homogeneous term in its graph near the origin, under certain geometric condition, can be reduced---up to biholomorphism of $\cplx^2$---to a one parameter family of the form
\[
M_t:=\left\{(z,w)\in\cplx^2: w=z^2\ba{z}+tz\ba{z}^2+\dfrac{t^2}{3} \ba{z}^3+o(|z|^3)\right\},\;\; t\in (0,\infty),
\]
near the origin.

We prove that $M_t$ is not locally polynomially convex if $t<1$. The local hull contains a ball centred at the origin if $t<\sqrt{3}/2$. We also prove that
$M_t$ is locally polynomially convex for $t\geq \sqrt{\dfrac{3}{2}}$.  
We show that, for $\sqrt{3}/2\leq t<1$, the 
polynomial hull of $M_t\cap \ba{B(0;\dl)}$ contains a one parameter family of analytic discs passing through the origin for every $\dl>0$.
We also prove that, if we remove the higher order terms from the graphing function of $M_t$, it
 is locally polynomially convex for $t\geq\dfrac{\sqrt{15-\sqrt{33}}}{2\sqrt{2}}$. 
Some new results about the local polynomial convexity of the union of three totally-real planes are also reported.
\end{abstract}

\maketitle

\section{Introduction and statement of results}\label{S:intro}
Let $K$ be a compact subset of $\cplx^n$. The polynomially convex hull of
$K$ is defined by $\hull{K}:= \left\{z\in \cplx^n : |p(z)|\leq \sup_{K}|p|,\, p\in \cplx [z_1, \dots, z_n]
 \right\}$. The set $K$ is said to be polynomially convex if $\hull{K}=K$. A closed subset $E$ of
$\cplx^n$ is said to be locally polynomially convex at $p\in E$ if $E\cap \ba{B(p;r)}$ is polynomially
convex for some $r>0$, where $B(p;r)$ denotes the open ball in $\cplx^n$ with centre $p$ and radius $r$.
In $\cplx$, polynomial convexity of a compact set $K$ is equivalent to $\cplx\setminus K$ is connected which is purely a topological condition on the compact set.
In higher dimensions, there are no such criterion. Polynomial convexity is one of the fundamental and classical concept in several complex variables; it came up, mainly, due to its deep interconnections with polynomial approximations. 
We will state a couple of 
such results for motivations. The first result is a generalization of Runge's approximation theorem. 
\begin{result}[Oka-Weil]\label{thm-OkaWeil}
Let $K\subset\cplx^n$ be a compact polynomially convex set. Then any function that is holomorphic 
in a neighborhood of $K$ can be approximated uniformly on $K$ by polynomials in $z_1,\dots, z_n$.
\end{result}

\noindent Few notations: For a compact $K\subset\cplx^n$, $\smoo(K)$ denotes the algebra of all continuous functions on $K$, and
$\poly(K)$
denotes the uniformly closed subalgebra of $\smoo(K)$ generated by polynomials 
in $z_1,\dots, z_n$. In this paper, planes will mean $2$-dimensional real subspaces of $\cplx^2$ and 
surfaces will be $2$-dimensional real submanifolds of $\cplx^2$

\smallskip

\noindent A real submanifold $M$ of $\cplx^n$ is said to be totally real at $p\in M$ if $T_pM\cap iT_pM=\{0\}$, where 
$T_pM$ denotes the tangent space of $M$ at $p$ viewed as a subspace of $\cplx^n$. The manifold $M$ is said to be totally real if it 
is totally real at every point $p\in M$. A point $p$ in a real surface $M$ in $\cplx^2$ is said to be a CR singularity if 
$M$ is not totally real at $p$, i.e., $T_pM$ is a complex subspace of $\cplx^2$.
 We now state another approximation result, which 
is a generalization of Stone-Weierstrass theorem.

\begin{result}[O'Farrell-Preskenis-Walsh]\label{thm-OPW}
Let $K\subset\cplx^n$ be a compact polynomially convex subset of $\cplx^n$ and $E\subset K$ be such that 
$K\setminus E$ is locally contained in totally-real submanifolds of $\cplx^n$. Then 
\[
\poly(K)=\left\{f\in\smoo(K): f|_E\in\poly(E)\right\}.
\]
\end{result}

\noindent If $E$ is an arc or finitely many points in $K$, $\hull{K}=K$ implies $\poly(K)=\smoo(K)$. The subsets that we will consider in this paper will be totally real except at one point.
\smallskip

 For a compact subset $K\subset \cplx^n$, the main obstruction to $\poly(K)=\smoo(K)$ is having analytic structure in $\hull{K}$. 
One of the ways to get analytic structure is by attaching analytic discs to $K$, that is, existence of a continuous map
$\phi:\ba{\disc}\to\cplx^n$ that is holomorphic on $\disc$ and $\phi(\bdy\disc)\subset K$. Constructions of  such families of analytic discs attached to a submanifold  are crucial for the hull of holomorphy and extension of holomorphic functions \cite{BedGav, BedKling, EB,  HuntWell, KanSten, KenWeb}.

We now give a brief survey about the local polynomial convexity of real submanifolds in $\cplx^2$. We start with  a very simple case: the graph of a linear function on $\cplx$ is either a 
totally-real subspace or a complex subspace that is of the form
$M=\{(z,f(z)\in \cplx^2: f(z)=az+b\ba{z}\}$. Any compact subset of $M$ is polynomially convex. 
$M$ is totally real if and only if $b\neq 0$. Next, we consider 
$\smoo^2$-smooth real surface in $\cplx^2$. For each point $a\in M$, there exists a $\delta>0$ such that 
\[
M\cap B(a;\delta) =\left\{(z,f(z))\in\cplx^2: f\in \smoo^2(D(a;\delta))\right\}. 
\]
Such a surface M is totally real at $a\in M$ if and only if $\dfrac{\bdy f}{\bdy \ba{z}}(a)\neq 0$. Wermer \cite{Wermer1} showed that a real surface $M$ in $\cplx^2$ is locally polynomially convex at each of its totally-real points, which was, then generalized to higher dimensions in \cite{HW}.
Let $0\in M$ be an isolated CR singularity  of the surface $M$. Locally, near the origin, $M$ is of the form
\begin{equation}\label{eq-surface}
 M\cap B(0;\delta) =\left\{(z,f(z))\in\cplx^2: f(z)=p_k(z,\ba{z})+ o(|z|^k)\right\},
 \end{equation}
 where $p_k(z,\ba{z})$ is a homogeneous nonholomorphic polynomial of degree  $k$, $k\geq 2$. 
 We say that such a surface $M$ has a CR singularity of {\em order k}.
 Bishop \cite{EB} considered the case $k=2$, i,e., $p_2(z,\ba{z})= az^2+b\ba{z}^2+cz\ba{z}$.
 Under Morse theoretic nondegeneracy condition (equivalently $c\neq 0$), Bishop \cite{EB} showed that, by biholomorphic change of variables, for sufficiently small $\dl>0$, $M\cap B(0;\dl)$ can be taken 
 to now-famous Bishop's normal form:
 \begin{equation}\label{eq-bishop}
 M\cap B(0;\delta) =\left\{(z,f(z))\in\cplx^2: f(z)= \gamma (z^2+\ba{z}^2)+|z|^2+ o(|z|^2)\right\},
 \end{equation}
 where $\gamma\geq 0$ is a biholomorphic invariant. The origin 
 is called the elliptic CR singularity if $\gamma<1/2$, parabolic if $\gamma=1/2$, and hyperbolic 
 if $\gamma>1/2$.
For $\gamma<1/2$, Bishop \cite{EB}   
  showed, by constructing  a one parameter family of analytic discs with boundaries passing around the origin in $M$,
 that the surface is not locally polynomially convex at the origin. Kenig-Webster \cite{KenWeb}, for $\smoo^\infty$-smooth surface $M$, showed that the local hull is a three dimensional $\smoo^\infty$-smooth manifold. For $0<\gamma<1/2$,
 Moser-Webster \cite{MosWeb} described complete local invariants of M in case $M$ is real analytic near the origin
 and it is immediate from their normal form that the local hull is a three dimensional real analytic manifold in this case. 
 Forstneri\v{c} and Stout \cite{FS} 
 showed that the surface is locally polynomially convex at the hyperbolic CR singularities. J\"{o}ricke \cite{J1}
 studied the case $\gamma=1/2$.
 For $k\geq 3$, the Morse theory does not give any nondegeneracy condition. This type of CR singularity is called 
 {\em degenerate} CR singularity. Before proceeding further with the discussion, we mention  the following definition on nonparabolic CR singularity \cite{GB3, GH3} that makes sense in case of degenerate CR 
 singularity as well as nondegenerate.
 \begin{definition}
 A surface $M$ of the form \eqref{eq-surface} is said to have a {\em nonparabolic} CR singularity at the origin if there is an isolated CR singular point for the surface $\{(z,p_k(z,\ba{z})): z\in\cplx\}$ at the origin, where $p_k$ 
 is as in \eqref{eq-surface}.
 \end{definition} 
\noindent  Harris \cite{GH3} showed that having a nonparabolic singularity of order $k$ at the origin
  is stable under $o(|z|^{k})$ perturbation.   
 We note that the quadratic terms in Bishop's normal form \eqref{eq-bishop} are real valued. 
 In general, for $k\geq 3$, there is no biholomorphic change of coordinates under which the lowest order homogeneous term becomes real valued. The question arises: {\em Is it possible to characterize the local polynomial convexity of a surface with nonparabolic CR singularity of order $k$, $k\geq 3$?}
 Efforts \cite{GB3, GH3} have been made to achieve a Bishop-type dichotomy for nonparabolic points in case of higher order CR singular points, but one of the assumptions of the results in this directions is---up to a biholomorphic change of variables---the lowest order homogeneous term is real valued. 
 Maslov-type index (see Subsection~\ref{ss-maslov} for definition) plays a crucial role in case of higher order degeneracy (see \cite{GB3}). 
 The papers \cite{GB1,GB2} describe conditions, in terms of the coefficients bound of the lowest order homogeneous term (not necessarily real valued), under which $M$ is locally polynomially convex at the origin.
 Wiegerink \cite{Wieg}, on the other hand, demonstrated 
 conditions under which the local hull is nontrivial (see Results~\ref{res-wiegerinck1}, \ref{res-wiegerinck} and \ref{res-wiegnonhomo}). 
 \smallskip
 
In this paper, we
restrict our attention to surfaces $M$ of the form \eqref{eq-surface} with $k=3$ and, to simplify the notation, from now onwards we call the polynomial $p_3$ as $p$. A general form of $p$ is:
 \[
 p(z,\ba{z})= a_1z^2\ba{z}+a_2z\ba{z}^2+a_3\ba{z}^3, 
 \]
 where $a_1, a_2, a_3\in\cplx$.
 We denote $S:=\{(z,p(z,\ba{z}))\in\cplx^2: z\in\cplx\}$.
 We impose the following geometric condition on $M$:
 
 $(*)$ {\em $M$ is a real surface in $\cplx^2$ as in \eqref{eq-surface} with $k=3$ and $S$ is as defined above from $M$ such that 
 there exists a proper holomorphic map $\Phi:\cplx^2\to\cplx^2$ with $\Phi^{-1}(S)$ is the union of three totally real planes 
 in $\cplx^2$.}
 \smallskip
 

 \noindent It is interesting to note that Condition $(*)$ turns out to be a condition on the coefficients of the polynomial $p$. We consider the
 proper holomorphic map $\Phi:\cplx^2\to\cplx^2$ defined by 
\begin{equation}\label{eq-defPhi}
\Phi(z,w):=(z, p(z,w)), 
\end{equation}
where $p(z,w)=a_3w^3+a_2w^2z+a_1wz^2$.
The following result, due to Thomas \cite[Lemma~3]{T1}, gives a necessary and sufficient condition in terms of the coefficients of $p$ 
for $\Phi^{-1}(S)$ to be a union of three totally-real planes. 

\begin{result}[Thomas]\label{lem-planes}
 Assume $S=\{(z,w)\in\cplx^2: w=p(z,\ba{z})\}$, 
where $p(z,w)=a_3w^3+a_2w^2z+a_1wz^2$.
Let $\Phi:\cplx^2\longrightarrow\cplx^2$ be as above. 
Then $\Phi^{-1}(S)$ is the union of three totally-real planes if and only if 
$a^2_2=3a_1a_3$ and $a_3\neq 0$.
\end{result}
 
\noindent In view of Result~\ref{lem-planes}, we obtain that, under the condition $a_2^2=3a_1a_3$ and $a_3\neq 0$,
the pre-image of $S$ under the particular proper holomorphic map $\Phi:\cplx^2\to\cplx^2$ 
is a union of three transverse totally-real planes. 
One is tempted to think that there might be different proper holomorphic map on $\cplx^2$ such that the 
pre-image under that map is also a union of three transverse totally-real planes, possibly different.  We show that, 
up to an invertible 
  complex linear map, the proper holomorphic map $\Phi$ 
is unique (see Lemma~\ref{lem-planesunique}). Hence, in view of Result~\ref{lem-planes}, 
for each $a\in\cplx\setminus \{0\}$, the given surface is 
 \begin{equation}\label{eq-graphThomas}
 S_a=\{(z,p_a(z,\ba{z}))\in\cplx^2:z\in\cplx\},
 \end{equation}
 where $p_a(z,\ba{z})=z^2\ba{z}+az\ba{z}^2+\dfrac{a^2}{3}\ba{z}^3$. 
 The complex number $a$ which appears in the coefficients of the cubic polynomial $p_a$ is not a biholomorphic invariant. Our search of a biholomorphic invariant for this family of surfaces leads us to certain normal form
 due to Haris\cite{GH1}. 


\smallskip


\begin{result}[Harris]\label{lem-equiv-surf}
Let $M$ is a $\smoo^\infty$-smooth real surface in $\cplx^2$ with an isolated CR-singularity of order $3$ at the origin. Then, locally, near the origin, up to biholomorphism of $\cplx^2$,  $M$ is of the form 
\[
\left\{\left(z, p(z,\ba{z})+o(|z|^3)\right)\in\cplx^2: |z|<\dl\right\},
\]
where $p$ is  one of the following degree $3$ homogeneous polynomials:
$z^2\ba{z}, z\ba{z}^2, \ba{z}^3$, $z^2\ba{z}+\gamma z\ba{z}^2, z^2\ba{z}+\gamma z\ba{z}^2+c\ba{z}^3$ with $\gamma>0$ and $c\in\cplx$. Moreover, $\gamma$ and $c$ are biholomorphic invariant.
 \end{result}
 Our geometric condition $(*)$ forces the normal form of $p$ to be of the form 
 $z^2\ba{z}+\gamma z\ba{z}^2+c\ba{z}^3$. Therefore,  by Result~\ref{lem-equiv-surf}, the normal form of surfaces of our study is 
 \begin{equation}\label{eq-normsurf}
 \{(z,p_t(z,\ba{z}))\in\cplx^2 : z\in\cplx\},\;\; t\in (0,\infty),
 \end{equation} 
 where 
 \[
 p_t(z,\ba{z})= z^2\ba{z}+tz\ba{z}^2+\dfrac{t^2}{3}\ba{z}^3.
 \]
  The real number $t$ becomes a biholomorphic invariant. Therefore,
 it is enough to consider the surfaces 
 $S_t$ for $t\in (0,\infty)$.  
 \smallskip
 
 We now provide a couple of argument supporting our claim that the family of surfaces $M_t, t\in (0,\infty)$, though the lowest degree homogenous term of the graphing function is not necessarily real valued, is a 
 right class of surfaces that can provide the Bishop-type dichotomy.  Firstly,
 Condition $(*)$ is also hidden in Bishop's surfaces. For every surface with nonparabolic CR-singularity, the graph of the quadratic polynomial in Bishop's normal form can be pulled back by proper holomorphic map on $\cplx^2$ to union of two totally-real planes. This leads to pulling back Bishop's surfaces with CR singularity to unions of two totally-real surfaces.
 Forstneri\v{c}-Stout \cite{FS}  used this to show local polynomial convexity at 
 hyperbolic CR singularity. Approaching to local polynomial convexity at CR singularity of higher order, in general, is difficult; this is the main reason of assuming `thin' or `flat' surfaces, i.e., the lowest order homogeneous terms in the graphing function to be real valued in \cite{GH3, GB3}. In 
 \cite{GB1, GB2}, pulling back by proper holomorphic mapping on $\cplx^2$ to union of certain totally-real surfaces is used crucially. Hence, Condition~$(*)$ is very natural condition to assume if we use this same approach for our surfaces of consideration. 
 Secondly, we look at Bishop's
 normal form in a little different viewpoint, using the following  biholomorphic transformations 
 $\tau_j:\cplx^2\to \cplx^2$, $j=1,2$, defined
 by 
 \begin{align*}
 \tau_1(z,w) &:=(z, w-\gamma z^2)\\
 \tau_2(z,w) &:=(z,4\gamma w)\\
 \tau_3(z,w)&:= (z, w+z^2).
 \end{align*}
 Define $\varphi:\cplx^2\to \cplx^2$ by 
 $
 \varphi(z,w):=\tau_3\circ\tau_2\circ\tau_1$.
 Therefore, we have
 \[
 \varphi(z,f(z))= (z+2\gamma \ba{z})^2+o(|z|^2). 
 \]
 By putting $t:=2\gamma$, we obtain that 
 \begin{equation}\label{eq-normalf}
 \varphi(z,f(z))=(z+t\ba{z})^2+o(|z|^2), 
 \end{equation}
 where $t$ is also a biholomorphic invariant. Surfaces with parabolic CR singularity are given by $t=1$.  
 From the relation of $t$ with Bishop's invariant, we can say that the surface has an elliptic  CR singularity if $t<1$, and hyperbolic CR singularity if $t>1$.  
 We now apply the biholomorphic map $\sigma:\cplx^2\to\cplx^2$ on $M_t$ defined by 
 \[
 \sigma(z,w):=\sigma_1\circ\sigma_2(z,w),
 \]
 where
 \begin{align*}
 \sigma_1(z,w)&:= (z, w+z^3)\\
 \sigma_2(z,w)&:=(z,3tw).
 \end{align*}
 We also see that the surface $M_t$ is, 
 locally near the origin, equivalent to a surface of the form: 
 \[
 \left\{(z,w)\in \cplx^2: w=(z+t\ba{z})^3+o(|z|^3)\right\}, \;\; t\in (0,\infty).
 \]
 Therefore, for obtaining a Bishop-type phenomenon,  a right class of surfaces with isolated CR singularity of order 
 $k$ at the origin to 
 consider are:
  \[
 \left\{ (z,w)\in\cplx^2: w=(z+t\ba{z})^k+o(|k|^k)\right\}, \;\; t\in (0,\infty).
 \]
\smallskip

 In this paper, we will study the surfaces for $k=3$, i.e., the surfaces $M_t,\; t\in (0,\infty)$. 
 For each $t\in(0,\infty)$, we will consider the corresponding proper holomorphic map 
 \begin{equation}\label{eq-propmap}
 \Phi_t:\cplx^2\to\cplx^2\;\;\text{defined by}\;\; \Phi_t(z,w)=(z, p_t(z,w)).
 \end{equation}
 Thomas \cite{T1} considered class of triples of totally-real planes that are the preimage of surfaces $S_a$ of the form \eqref{eq-graphThomas}, where $a\in\cplx$ sufficiently small, for demonstrating triples of totally-real planes whose pairwise unions are locally polynomially convex at the origin but the local hull of the whole union contains a ball centred at the origin and with positive radius. We note that, for every $t\in\cplx$, the map
 $\Phi_t$, defined in \eqref{eq-propmap}, is a proper holomorphic map on $\cplx^2$. For each $t\in\cplx\setminus\{0\}$, the pre-image $\Phi_t^{-1}(S_t)$ is a union of three totally-real planes. Thanks to the proper map \eqref{eq-propmap} and
 the normal form \eqref{eq-normsurf} of $S_t$, it is enough to consider the following triples of totally-real planes whose image is $S_t$, $t\in (0,\infty)$, under the proper holomorphic map \eqref{eq-propmap}. 
\begin{align}
P_0^t&: \qquad w=\ba{z},\notag\\
P_1^t&: \qquad w=-\dfrac{\sqrt{3}(\sqrt{3}-i)}{2t} z-\dfrac{1-i\sqrt{3}}{2}\ba{z},\notag\\
P_2^t&:\qquad w=-\dfrac{\sqrt{3}(\sqrt{3}+i)}{2t} z-\dfrac{1+i\sqrt{3}}{2}\ba{z}.\label{eq-planes-graph}
\end{align}
Our first couple of results are about the local polynomial convexity of compact subsets of $P_0^t\cup P_1^t\cup P_2^t$.
\begin{theorem}\label{thm-3planes}
For each $t>\dfrac{\sqrt{15-\sqrt{33}}}{2\sqrt{2}}$, $P_0^t\cup P_1^t\cup P_2^t$ is locally polynomially convex at the origin.
\end{theorem}
\noindent The first part of the following theorem gives a generalization of a result of Thomas \cite{T1}, and the second part demonstrates new phenomenon for the union of 
three totally-real planes in $\cplx^n$.
\begin{theorem}\label{thm-planes-hull}
If $t<1$, $P_0^t\cup P^t_1\cup P^t_2$ is not locally polynomial convex at the origin. Moreover, 
for every $\dl>0$, 
\begin{itemize}
\item[(i)] the polynomial hull of $(P_0^t\cup P_1^t\cup P_2^t)\cap \ba{B(0;\delta)}$ contains 
a neighbourhood of the origin in $\cplx^2$ if $0<t<\sqrt{3}/2$; and 

\item[(ii)] the polynomial hull contains a one parameter family of analytic varieties passing through the origin if $\sqrt{3}/2\leq t<1$.
\end{itemize}
\end{theorem}

The next couple of results are about the local polynomial convexity of compact subsets of $S_t$
at the origin. 
 \begin{theorem}\label{thm-cubic-pcvx}
For $t\in(0,\infty)$, let $S_t=\{(z,w)\in\cplx^2:w=p_t(z,\ba{z})\}$, where $p_t(z,\ba{z})=z^2\ba{z}+tz\ba{z}^2+t^2\ba{z}^3/3$. Then
\begin{itemize}
\item[i)] the surface $S_t$ is locally polynomially convex at the origin if $t>\dfrac{\sqrt{15-\sqrt{33}}}{2\sqrt{2}}$
\item[ii)] $S_1$ is locally polynomially convex at the origin.
\end{itemize}
\end{theorem}

\noindent A couple of new phenomena occur here which were not present in case of Bishop's surfaces. Part $(ii)$ 
of the next theorem describes the surfaces $S_t$ whose local hull contains a nonempty open ball centred at the origin and Part $(iii)$ determines those surfaces which has analytic discs with boundary in $S_t$ and passes through the CR singularity at the origin.
\begin{theorem}\label{thm-cubic-hull} 
Let $S_t=\{(z,w)\in\cplx^2:w=p_t(z,\ba{z})\}$, where $p_t(z,\ba{z})=z^2\ba{z}+tz\ba{z}^2+t^2\ba{z}^2/3$. 
Then 
\begin{itemize}
\item[i)] For $t\in(0,1)$, $S_t$ is not locally polynomially convex at the origin.
\item[ii)] For every $\delta>0$, the polynomial hull of $S_t\cap \ba{\mathbb{B}(0;\delta)}$ contains a nonempty open ball centred at the origin if $t\in (0,\sqrt{3}/2)$.
\item[ii)] For every $\dl>0$, the polynomial hull of $S_t\cap \ba{\mathbb{B}(0;\delta)}$ contains a one parameter family of analytic discs with boundary in $S_t$ passing through the origin if $t\in[\sqrt{3}/2, 1)$ 
\end{itemize}
\end{theorem}

We now consider the surfaces $M_t\subset \cplx^2$ with an isolated
CR singularity at the origin of order three, i.e. there exist 
$r>0$ such that 
\begin{equation}\label{eq-gensurf}
M_t\cap \ba{B(0;r)}=\{(z,w)\in\cplx^2: \phi(z,w)=0\},
\end{equation}
where $\phi(z,w)= p_t(z,\ba{z})+F(z,\ba{z})-w$ with $F(z)=o(|z|^3)$.
Next, we state a couple of theorems for $M_t$ analogous to Theorems~\ref{thm-cubic-pcvx} and \ref{thm-cubic-hull}.

\begin{theorem}\label{thm-surface-pcvx}
For $t\in (0,\infty)$,  let $M_t$ be as in \eqref{eq-gensurf}. Then $M_t$ is locally polynomially convex 
at the origin if
	 $t\geq \sqrt{\dfrac{3}{2}}$.
\end{theorem}

\begin{theorem}\label{thm-surface-hull}
For $t\in (0,\infty)$,  let $M_t$ be as in \eqref{eq-gensurf}. Then
\begin{itemize}
\item[i)] $M_t$ is not locally polynomially convex if $t\in (0,1)$. 
	 \item[ii)] For every $\dl>0$, the polynomially convex hull of $M_t\cap \ba{\mathbb{B}(0;\dl)}$ contains a nonampty open ball centred at the origin if $t\in (0, \sqrt{3}/2)$.
	 \item[iii)] For every $\dl>0$, the polynomial hull of $M_t\cap \ba{\mathbb{B}(0;\delta)}$ contain 
	 a one parameter family of analytic discs with boundary in $M_t$ and passing through the origin if $\sqrt{3}/2\leq t<1$ .
\end{itemize} 
\end{theorem}

\begin{remark}We make a couple of remarks here.
\begin{itemize}
\item [(i)] The number $\sqrt{3}/2$ has a crucial geometric meaning. For this number onward the boundary of any 
analytic disc attached to $M_t$ have to pass through the CR singularity, the origin. It seems a hyperbolic sector is appearing in the picture. In another viewpoint, both the matrices corresponding to Weinstock's normal 
form (see Lemma~\ref{lem-normformplane}) of the planes have nonreal eigenvalues if and only 
if $0<t<\sqrt{3}/2$.

\item[(ii)] The numbers $\sqrt{\dfrac{3}{2}}$ and $\dfrac{\sqrt{15-\sqrt{33}}}{2\sqrt{2}}$ seem to have no special importance. It is just that our technique does not work. We expect that $M_t$ will be locally polynomially convex at the origin for $t>1$.
\end{itemize}
\end{remark}
\smallskip

We now turn our discussion towards local polynomial convexity and the local hull of the unions of three totally-real surfaces. 
In view of Condition~$(*)$, the unions of three totally-real surfaces have a very close connection with the surfaces $M_t$. One of the main approaches of showing local polynomial convexity of surfaces at the isolated CR singularity comes from looking at the unions totally-real surfaces those arise as pre-images under a proper holomorphic maps, and showing that the unions are locally polynomially convex at the origin. Forstneri\v{c}-Stout \cite{FS} 
used this to prove their theorem. We will also use this in our proofs of Theorem~\ref{thm-cubic-pcvx} and Theorem~\ref{thm-surface-pcvx}. 
 If the totally-real surfaces are pairwise transverse, then locally the union of totally-real surfaces can be seen as a union of small perturbations of their tangent spaces at the origin.
The study of the local polynomial convexity of the union of two totally-real subspaces began with Weinstock's \cite{Wk} (see Result~\ref{thm-weinstock}) necessary and sufficient condition.  The unions of three totally-real planes in $\cplx^2$ was first considered by Thomas \cite{T1, T2} showing that
 there exist triples of totally-real planes such that the local hull of the union contains a ball centred at the origin 
and, also, by demonstrating examples where there are no nontrivial hull. 
Sufficient conditions for local polynomial convexity of the unions of three totally-real planes in $\cplx^2$, 
in terms of the pair of matrices in Weinstock's normal form (see Lemma~\ref{lem-normformplane}), is given in \cite{SG2} (see Result~\ref{res-3planes}).
In this paper, we provide some new results in this setting (see Theorem~\ref{thm-3planesspl}, Theorem~\ref{thm-planesdetlessone}, Theorem~\ref{thm-planeszw}) that will also be used in the proof of Theorem~\ref{thm-cubic-pcvx}. 
It is not trivial to pass the local polynomial convexity to the union of totally-real submanifolds intersecting at the 
origin from the union of their tangent spaces at the origin. 
The union of two totally-real submanifolds in $\cplx^n$ intersecting transversely at the origin is locally polynomially convex if the union of their tangent spaces is \cite{SG1, ShafSukh}. These results use Weinstock's result 
(Result~\ref{thm-weinstock}) crucially. In our case Kallin's lemma will allow us to pass local polynomial convexity from the union of tangent spaces to the union of totally-real surfaces for $t\geq\sqrt{\dfrac{3}{2}}$ as in \cite{SG1}. This passage does not seem to work for $\dfrac{\sqrt{15-\sqrt{33}}}{2\sqrt{2}}<t\leq \sqrt{\dfrac{3}{2}}$.


Few comments about our proofs of the theorems. 
\begin{itemize}
\item In the base of our approach towards local polynomial convexity of $M_t$, there lies the union of
 three pairwise transverse totally-real planes. Therefore, we first prove Theorem~\ref{thm-3planes} with the help 
of Result~\ref{res-3planes} and theorems that are stated and proved in Section~\ref{sec-unionplanes}. 
The first part of Theorem~\ref{thm-cubic-pcvx} follow immediately from Theorem~\ref{thm-3planes}. For $S_1$, the planes which we get are not pairwise transverse. We deal with that separately.
 
\item To prove Theorem~\ref{thm-surface-pcvx}, we use Kallin's lemma to enable 
us to prove local polynomial convexity of union of three totally real surfaces which can be seen as 
small perturbation of the planes in Theorem~\ref{thm-3planes}. 

\item We use some results of Wiegerink (Results~\ref{res-wiegerinck1}, \ref{res-wiegerinck} and \ref{res-wiegnonhomo}) to prove Theorems~\ref{thm-cubic-hull} and \ref{thm-surface-hull}. 
Theorem~\ref{thm-planes-hull} follows from Theorem~\ref{thm-cubic-hull}.
\end{itemize}
\smallskip

About the layout of the paper: In Section~\ref{sec-techpreli}, we collect some results from literature
 and prove a lemma about uniqueness of proper holomorphic map, up to invertible $\cplx$-linear 
 transformation. We also provide a 
discussion about Maslov-type index here.  
We state and prove three general results about local polynomial convexity of union of three totally-real 
planes in $\cplx^2$ in Section~\ref{sec-unionplanes}.
 In Section~\ref{sec-polyperturb}, we state and prove that the local polynomial convexity of certain unions of three totally real surfaces. Proofs of Theorem~\ref{thm-cubic-pcvx} and Theorem~\ref{thm-surface-pcvx} and Theorem~\ref{thm-3planes}
are discussed in Section~\ref{sec-proofs-pcvx}. We demonstrate the proofs of  Theorem~\ref{thm-cubic-hull}, Theorem~\ref{thm-surface-hull} 
and Theorem~\ref{thm-planes-hull} in Section~\ref{sec-surface-hull}. Some open questions are mentioned in
Section~\ref{sec-questions}.

\section{Preliminaries}\label{sec-techpreli}

The most part of this section is devoted to collect the results from the literature that will be used 
in our proofs. We begin the section by showing that the proper map in Condition $(*)$ is unique up to 
invertible $\cplx$-linear transformations. 
\begin{lemma}\label{lem-planesunique}
Let $S=\{(z,p(z,\ba{z}))\in\cplx^2: z\in\cplx \}$, where 
$p(z,\ba{z})=\sum_{j=1}^3 a_j\ba{z}^jz^{3-j}$ with $a_3\neq 0$ and $a_2^2=3a_1a_3$. 
Let $\Psi:\cplx^2\to\cplx^2$ be 
a proper holomorphic map of the form $\Psi(z,w)=(z,Q(z,w))$, where $Q$ 
is a homogenous polynomial. 
Assume further that $\Psi^{-1}(S)$ can be written as union of 
three pairwise transverse totally-real planes. Then $\Psi=\Phi$ as defined in \eqref{eq-defPhi} up to a $\cplx$-linear transformation. 
\end{lemma}

\begin{proof}
We first note that the degree of the polynomial $Q$ must be three, and, since $Q$ 
is homogenous, it can be written as: 
\[
Q(z,w)=c(\mu_1z+w)(\mu_2z+w)(\mu_3 z+w).
\]
Then, if we designate $\Psi^{-1}(S)=\cup_{j=1}^3\{(z,w)\in\cplx^2: w=\alpha_jz+\beta_j\ba{z}\}$, 
for $\al_j,\beta_j\in\cplx, \;j=1,2,3$, we must investigate the solutions $(\alpha,\beta)$ of the equation 
\begin{equation}\label{eq-cubicto3planes}
c\prod_{j=1}^3 ((\mu_j+\alpha)z+\beta \ba{z})=\sum_{j=1}^3a_j\ba{z}^jz^{3-j}.
\end{equation}
Comparing the coefficients, we get that $c=a_3$ and 
\begin{align*}
(\mu_1+\alpha)(\mu_2+\alpha)(\mu_3+\alpha)&=0,\\
(\mu_1+\mu_2+\mu_3+3\alpha)&=\dfrac{a_2}{a_3\beta^2},\\
(\mu_1+\alpha)(\mu_2+\al)+(\mu_2+\al)(\mu_3+\al)+(\mu_3+\al)(\mu_1+\al)&=\dfrac{a_1}{a_3\beta},\\
\beta^3 &=1.
\end{align*}
Thus, we have $\beta\in\{1,\omega,\omega^2\}$, where $\omega=e^{2\pi i/3}$. Let us name the pairs of 
coefficients as $(\alpha_j,\beta_j)$, $j=1,2,3$ defining the three planes, where $\beta_1=1$, $\beta_2=\omega$, 
and $\beta_3=\omega^2$. We note that, for fixed 
$k\in\{1,2,3\}$,  $\mu_j+\alpha_k$, $j=1,2,3$, are the roots of the following equation:
\[
\lambda^3-\dfrac{a_2}{a_3} \left(\dfrac{\lambda}{\beta_k}\right)^2+\dfrac{a_1\lambda}{a_3\beta_k}=0.
\]
Consider the equation 
\begin{equation}\label{eq-poly}
r^3-\dfrac{a_2}{a_3}r^2+\dfrac{a_1}{a_3}r=0.
\end{equation}
By previous computations, we see that each of the following triples are the roots of Equation~\eqref{eq-poly}:
$
(\mu_1+\al_1, \mu_2+\al_1, \mu_3+\al_1), (\omega(\mu_1+\al_2), \om(\mu_2+\al_2), \om(\mu_3+\al_2)), 
(\om^2(\mu_1+\al_3),\om^2(\mu_2+\al_3), \om^2(\mu_3+\al_3)).$ 
If we assume that $r_1,r_2$ and $r_3$ are the roots of Equation~\eqref{eq-poly},
then the polynomial $Q$ exists such that $\Psi$ is 
a proper map with the required properties if and only if there exist permutations 
$\pi_1,\pi_2\in S_3$ such that 
\begin{align*}
r_k&=\omega^j r_{\pi_j(k)}+ \alpha_1-\alpha_{j+1},\qquad j=1,2, \quad k=1,2,3.
\end{align*}
We also note that one of the members in each triple 
is zero. Therefore, by looking at the first triple, we obtain that
\[
\{(z,w)\in\cplx^2: w=\ba{z}\} \subset \Psi^{-1}(S).
\]
This implies that 
\[
Q(z,\ba{z})=p(z,\ba{z})\qquad \forall z\in\cplx.
\]
Since $\{(z,w)\in\cplx^2:w=\ba{z}\}$ is maximally totally real, therefore, we conclude that
\[
Q(z,w)=p(z,w)\qquad \forall(z,w)\in\cplx^2.
\]
Choosing any other triple will contribute a $\cplx$-linear transformation to the map $\Phi$.
 \end{proof}

The following result \cite[Theorem~1.6.24]{stout} plays a vital role in this paper. 
\begin{result}\label{lem-propermap}
Let $\Phi:\cplx^n\longrightarrow\cplx^n$ be a proper holomorphic map. A 
compact subset $K$ of $\cplx^n$ is polynomially convex if and only if 
$\Phi^{-1}(K)$ is polynomially convex.
\end{result}

Next, we state a lemma due to Kallin \cite{K} (also see \cite{dP2}), which
will be used repeatedly in this sequel.
It gives a condition under which the union of two polynomially convex sets is polynomially convex.
\begin{result}[Kallin]\label{lem-Kallin}
 Let $K_1$ and $K_2$ be two compact polynomially convex subsets in $\cplx ^n$. Suppose
$L_1$ and $L_2$ are two compact polynomially convex subsets of $\cplx$ with
$L_1 \cap L_2 =\{0\}$. Suppose further that there exists a holomorphic polynomial
$P$ satisfying the following conditions:
 \begin{enumerate}
 \item[$(i)$] $P(K_1) \subset L_1$ and  $P(K_2) \subset L_2$; and
 \item[$(ii)$] $P^{-1}\{0\} \cap (K_1 \cup K_2)$ is polynomially convex.
 \end{enumerate}
 Then $K_1 \cup K_2$ is polynomially convex.
\end{result}
\smallskip

The next couple of lemmas are of linear algebraic flavour. The first one from \cite{SG2} gives
certain normal form for pair of matrices under similarity of matrices with real entries.
\begin{lemma}\label{lem-normal}
Let $A, B\in\rea^{2\times 2}$ such that $A$ has two 
distinct real eigenvalues, say $\lambda_1, \lambda_2$. Also assume that 
$\det[A,B]>0$. Then there exists an invertible matrix $T\in\rea^{2\times 2}$ 
such that 
\[
TAT^{-1}=\begin{pmatrix}
\lambda_1& 0\\
0 & \lambda_2
\end{pmatrix}
\quad\text{and}\quad
TBT^{-1}=\begin{pmatrix}
s_1 & q\\
q  & s_2
\end{pmatrix}
\]
for $\lm_j, s_j, q\in\rea$, $j=1,2.$
\end{lemma}
\noindent The next result from \cite{Wk} gives a normal form for tuple of totally-real maximal subspaces 
of $\cplx^n$.
\begin{lemma}\label{lem-normformplane}
Let $P_0,\dots,P_N$ be 
maximal totally-real subspaces in $\cplx^n$ such that $P_j\cap P_0=\{0\}$ for all $j=1,\dots, N$. Then the subspaces  can be 
 parametrized by a tuple of $n\times n$ matrices with real entries as follows: 
 \[
P_0=\rea^n,\;\; P_j=(A_j+i\id)\rea^n,\; j=1,\dots, N,
\]
where $A_j\in\rea^{n\times n}$, $j=1,\dots, N.$ 
\end{lemma}
\noindent We call this normal form as Weinstock's normal form.
Weinstock \cite{Wk} proved the following theorem about the union of two transverse totally-real 
subspaces in $\cplx^n$.

\begin{result}[Weinstock] \label{thm-weinstock}
 Suppose $P_1$ and $P_2$ are two totally-real subspaces of $\cplx^n$ of maximal dimension intersecting
only at $0 \in \cplx^n$. Denote the normal form for this pair as:
\[
 P_1 :=  \rea^n \;\text{and}\;
P_2 := (A+i \id)\rea^n,
\]
for some matrix $A$ with real entries. Then the union $P_1 \cup P_2$ is locally polynomially convex at the origin if and only if $A$ has no purely
imaginary eigenvalue of modulus greater than 1.
\end{result}
In this paper we are concerned only with the case $n=2$. We wish to transform the planes in \eqref{eq-planes-graph} into Weinstock's normal form. In view of Lemma~\ref{lem-normformplane}, Weinstock's normal form depends only on whether the intersection with a particular totally real subspace is trivial or not. In our case we consider that particular plane to be $P_0$. To transform into Weinstock's normal form it is enough to consider a couple of totally real planes of the form: 
\begin{align}
P_0&: \qquad w=\ba{z},\notag\\
P_1&: \qquad w=\al z+\beta \ba{z},\label{eq-generalplanesgraph}
\end{align}
where $\beta\neq 0$.
We first apply the following invertible $\cplx$-linear transformation
\[
T:\cplx^2\to\cplx^2\]
defined by 
\[
T(z,w)=(z+w, -i(z-w)).
\]
Clearly, $T(P_0)=\rea^2$ and 
\[
T(P_1)=\{(z+\al z+\beta\ba{z}, -i(z-\al z-\beta\ba{z})): z\in\cplx\}.
\]
Putting $z=x+iy$, we note that 
\[
T(P_1)=(B+iC)(\rea^2),
\]
where $B$ and $C$ are two $2\times 2$ matrices with real entries that we get as follows: putting $\al=\al_1+i\al_2$ 
and $\beta=\beta_1+i\beta_2$
\[
B=\begin{pmatrix}
 1+\al_1+\beta_1 & \beta_2-\al_2\\
 -\al_2-\beta_2 & 1-\al_1+\beta_1.
 \end{pmatrix}\;\;\text{and}\;\;
 C=\begin{pmatrix}
 \al_2+\beta_2 & 1+\al_1-\beta_1\\
 \al_1+\beta_1-1 & \beta_2-\al_2.
 \end{pmatrix}
 \]
 Here we view the matrices as linear transformations. We have $P_0\cap P_1=\{0\}$ implies that the matrix $C$ is invertible. The condition $P_0\cap P_1=\{0\}$ transfers into $|\beta|^2-|\al|^2-2\rl \beta +1 \neq 0$. 
 Hence, we can write 
\[
T(P_1)=(BC^{-1}+i\id)(C(\rea^2)).
\]
Since the matrix $C$ is invertible matrix with real entries, we obtain that
\[
T(P_1)=(BC^{-1}+i\id)(\rea^2).
\]
Therefore, Weinstock's normal form is $P_0 \sim \rea^2$ and $P_1\sim (A+i\id)(\rea^2)$, where 
$A=BC^{-1}$ as before. 
Putting the values of $\al$ and $\beta$ we obtain Weinstock's normal form for the triple of planes $(P_0^t,P_1^t,P_2^t)$ in \eqref{eq-planes-graph} as 
\[
P_0^t=\rea^2\;\text{and}\; P_j^t=(A_j^t+i\id)\rea^2, j=1,2, 
\]
where the matrices $A_1^t$ and $A_2^t$ are:
\begin{equation}\label{eq-planemat}
A_1^t=\begin{pmatrix}
\dfrac{t}{\sqrt{3}(1+t)} & \dfrac{1}{1+t}\\
-\dfrac{1}{1-t} & -\dfrac{t}{\sqrt{3}(1-t)}
\end{pmatrix}
\quad\text{and}\quad
A_2^t=\begin{pmatrix}
-\dfrac{t}{\sqrt{3}(1+t)} & \dfrac{1}{1+t}\\
-\dfrac{1}{1-t} & \dfrac{t}{\sqrt{3}(1-t)}
\end{pmatrix}.
\end{equation}
We will use these matrices in the proofs of our theorems.

Next, we mention a general result from \cite{SG2} about local polynomial convexity of the union of three totally-real planes in 
$\cplx^2$. 
It gives a sufficient condition, in terms of the matrices involved in the Weinstock's normal form, for local polynomial convexity of union of three totally-real planes at $0\in\cplx^2$. We need few notations from \cite{SG2}:
\[
\Omega:=\left\{(A_1,A_2)\in(\rea^{2\times 2})^2: \det[A_1,A_2]\neq 0, \# \sigma(A_1)=2, i\notin\sigma(A_j) \forall j \right\}
\]
\begin{align*}
\Theta(A_1,A_2) &:=\det A_1(\Tr A_2)^2+ \Tr A_1A_2(\Tr A_1A_2- \Tr A_1 \Tr A_2),\\
\Lambda(A_1,A_2) &:= 4\det A_1A_2 - \dfrac{1}{4}(\Tr A_1 \Tr A_2)^2.
\end{align*}

\begin{result}[Gorai]\label{res-3planes}
  Let $P_0,P_1, P_2$ be three totally-real planes in $\CC$ such that
$P_0 \cap P_j=\{0\}$ for $j=1,2$.
Hence, let Weinstock's normal form for $\{P_0,P_1, P_2\}$ be
\[
 P_0 = \RR,\;
 P_j =(A_j+i \id)\RR, \quad j=1,2,
\]
and assume $(A_1,A_2)$ belongs to parameter domain $\OM$.
 Assume further that the pairwise unions of $P_0,P_1,P_2$ are
locally polynomially convex at $0\in \CC$. Given $j\in \{1,2\}$, let $j^{{\sf C}}$ denote
the other element in $\{1,2\}$. Then:
\begin{itemize}
\item[$(i)$]Let $\sm(A_j)\subset \rea,\, j=1,2$. If
\begin{align*}
\text{either} &\; \;
  \det A_j\det[A_1,A_2]>0,\, j=1,2, \\
\text{or}&\;\; \det A_j\det [A_1,A_2]<0\;\; \text{and}\;\; (\det A_j)\Tht(A_j,A_{j^{{\sf C}}})<0 \;\;
\text{for some}\;\; j\in \{1,2\},
\end{align*}
then $P_0\cup P_1\cup P_2$ is locally polynomially convex at $0\in \CC$.

\item[$(ii)$] Suppose $\sm(A_1)\subset \rea$ and $\sm(A_2)\subset \cplx\setminus\rea$. If
\begin{align}
 \text{either}& \; \; \det A_1\det[A_1,A_2]<0 \;\;\; \text{and}\; \; \; (\det A_1)\Tht(A_1,A_2)<0 \notag \\
 \text{or}& \; \; \Tht(A_1,A_2)<\Lambda(A_1,A_2), \notag
\end{align}
then $P_0\cup P_1\cup P_2$ is locally polynomially convex at $0\in \CC$.
\smallskip

\item[$(iii)$] Suppose $\sm(A_j)\subset \cplx\setminus\rea,\, j=1,2$. If
$\Tht(A_j,A_{{ j^{\sf C}}})< \Lambda(A_1,A_2)$ for some $j\in \{1,2\}$,
then $P_0\cup P_1\cup P_2$ is locally polynomially convex at $0\in \CC$.
\end{itemize}
\end{result}

\subsection{Maslov type index}\label{ss-maslov}

Let $M$ be a $\smoo^1$-smooth orientable totally-real submanifold of $\cplx^2$ and 
$\gamma:S^1\to M$ be a curve. In this case, the pull back bundle $\gamma^* TM$ over $\gamma$ is a 
trivial bundle. Let $X_1, X_2$ be global sections of the above pull back bundle over $\gamma$ such that
\begin{equation}\label{eq-tgtsp}
T_{\gamma(\zeta)}M={\sf Span}_\rea\{X_1(\zeta), X_2(\zeta)\} \quad \forall \zeta\in S^1.
\end{equation}

\begin{definition}
Let $\gamma$ be curve in a $\smoo^1$-smooth totally-real submanifold of $\cplx^2$.
The {\em Maslov-type index of $\gamma$} in $M$ is denoted as $Ind_M(\gamma)$ and is defined as the winding number of the map
$h_{(X_1,X_2)}:S^1\to \cplx\setminus \{0\}$ defined by $h_{(X_1, X_2)}(\zeta):=\det(X_1(\zeta), X_2(\zeta))$, where $X_1, X_2$ are 
global sections of the above pull back bundle over $\gamma$ satisfying \eqref{eq-tgtsp}. 
 \end{definition}

\begin{remark}
We note that, for any other choice of a pair of global section $\{Y_1, Y_2\}$ of the pull back bundle $\gamma^* TM$ over $\gamma$, 
there exist a smooth map $A:S^1\to GL(n,\rea)$ such that $Y_j(\zeta)=A(\zeta)X_j(\zeta)$, $j=1,2$. Hence,
the winding number of the map $h_{(X_1,X_2)}$ is equal to the winding number of the map $h_{(Y_1,Y_2)}$.
Therefore, the Maslov-type index of $\gamma$ in $M$ is independent of the choice of the global sections of $\gamma^* TM$.
\end{remark}

\begin{remark}
We note that any two closed curves $\gamma_1$ and $\gamma_2$ which are homologous in $M$ have the same Maslov-type index (See \cite{FF1} for more discussions) 
\end{remark}

\noindent We now proceed to define the Maslov-type index of a surface in $\cplx^2$ with an isolated CR singularity.
\begin{definition}
Let $M$ be an oriented real submanifold of $\cplx^2$ with an isolated CR singularity at $p\in M$ and $U_p$ be 
a contractible neighbourhood of $p$ in M such that $U_p\setminus \{p\}$ is an oriented totally-real submanifold
of $\cplx^2$, where the orientation induced by the orientation of $M$. The {\em Maslov-type index of $p$} is denoted by $Ind_M(p)$ and is defined 
by the Maslov-type index of a simple closed curve $\gamma:S^1\to U_p\setminus\{p\}$ that winds around the point $p$.  
\end{definition}

\begin{remark}
Since $Ind_M(\gamma)$ depends on the homology class of $\gamma$ and the neighbourhood $U_p$ is contractible, we see that definition 
of Maslov-type index of an isolated CR singularity is independent of the curve $\gamma$ chosen.
\end{remark}

\noindent We now mention a lemma from \cite{FF1} which is pertinent to graphs over domains in $\cplx$ in $\cplx^2$.

\begin{result}\cite[Lemma 8]{FF1}\label{lem-maslovwinding}
Let $\Omega$ be a domain in $\cplx$ containing the origin and $f\in\smoo^1(\Omega)$ such that 
$\left(\dfrac{\partial f}{\partial \ba{z}}\right)^{-1}\{0\}=\{0\}$. Suppose that the 
graph $S_f$ has an isolated CR singularity at $0\in\cplx^2$. Let $\gamma:S^1\to\Omega\setminus \{0\}$ 
be a smooth, positively-oriented, simple closed curve that encloses $0\in \cplx^2$. Then 
the Maslov-type index $Ind_{S_f}(0)$ equals to the winding number of the curve $\dfrac{\partial f}{\partial \ba{z}}\circ \gamma$ 
around the origin.
\end{result}
\noindent Next, we mention a lemma due to Bharali \cite{GB3}, which gives an easy way to compute the Maslov-type index of 
a graph of homogeneous polynomial in $z$ and $\ba{z}$ of degree $k$ around an isolated CR singularity at the 
origin.

\begin{result}\label{res-maslov}\cite[Lemma 2.5]{GB3}
Let $p$ be a homogeneous polynomial of degree 
 $k$ in $z$ and $\ba{z}$ such that $\left\{z\in\cplx: \dfrac{\partial p}{\bdy \ba{z}}(z,\ba{z})=0\right\}=\{0\}$. 
 Define a polynomial $q$ in $z$ by $q(z)=\left.\dfrac{\bdy p}{\bdy \ba{z}}(z,\ba{z})\right|_{\ba{z}=1}$. Then the Maslov-type index of 
 the graph $S_p$ of $p$ of the origin,
 \[
 Ind_{S_p}(0)=2\left(\sum \left\{\mu(\zeta):\zeta\in q^{-1}\{0\}\cap \mathbb{D} \right\} \right)-(k-1),
 \]
where $\mu(\zeta)$ denotes the multiplicity of $\zeta$ as a zero of the polynomial $p$.
\end{result}

\noindent The next theorem is due to Forstneri\v{c} \cite{FF1} in $\cplx^2$, which says, in certain cases, existence of one analytic 
disc attached to a totally-real surface $M$ in $\cplx^2$ gives the existence of many analytic discs attached to it.

\begin{result}[Forstneri\v{c}]\label{res-andisc}
Let $\gamma\subset M$ be the boundary of an immersed analyitc disc $F_0:\ba{\disc}\to\cplx^2$ in a 
totally-real surface $M$ of $\cplx^2$ with Maslov-type index $j\geq 1$. Then $F_0$ lies in a $2j-1$ parameter family of analytic discs with boundary in $M$. If $\wtil{M}$ is a sufficiently small $\smoo^2$-perturbation of $M$, 
then there exists such a $2j-1$ parameter family of analytic discs with boundary in $\wtil{M}$. If $j=1$, the union of discs form a Levi flat hypersurface. If $j>1$, the union of the discs contains an open ball. 
\end{result}

\noindent Next, we state a series of theorems due to Wiegerinck \cite{Wieg} about the local polynomial hull of certain graphs. We will use these results in our proofs of describing local hulls.
\begin{result}[Wiegerinck]\label{res-wiegerinck1}
Let $\varphi$ be a $\smoo^k$-smooth, $k\geq 2$, function on a disc in $\cplx$
centred at the origin. Suppose that the graph of $\varphi$, denoted by $S_\varphi$, has an isolated CR singularity 
at the origin of Maslov-type index $j$, $0<j<k$.
If $\rl(\varphi(z,\ba{z})/z^{j-1})$ is strictly subharmonic 
on a punctured neighbourhood of the origin, then there exist analytic discs with 
boundary in $S_\varphi$.
\end{result}

\begin{result}[Wiegerinck]\label{res-wiegerinck}
Let $F(z,\ba{z})$ be a homogeneous function of degree $k$ in $z$ and $\ba{z}$, and 
$\smoo^2$-smooth away from origin in $\cplx$. Suppose that the origin is an isolated 
CR singularity of $S=\{w=F(z,\ba{z})\}$ and the Maslov-type index at $0\in\cplx^2$ 
is $j$, $0<j<k$. Assume that $\rl(F(z,\ba{z})/z^{j-1})$ is a subharmonic but nowhere harmonic function on 
$\cplx\setminus\{0\}$.  
Then
\begin{itemize}
\item[(i)] $S$ is not locally polynomially convex at $0\in\cplx^2$.

\item[(ii)] For evey $r>0$, the polynomial hull of $S\cap\ba{B(0;r)}$ contains a $(2j-1)$-parameter family of analytic discs with boundary in $S$ 
passing around zero if and only if the curve $\mathscr {C}:S^1\to \cplx$ defined by 
\[
\mathscr{C}(z)=\dfrac{F(z,\ba{z})}{z^k}
\]
has the following property:
${\sf (**)}$  If, for two different points $z_1\neq z_2$ on the unit circle, $\mathscr{C}(z_1)=\mathscr{C}(z_2)$, 
then $z_1$ and $z_2$ divide the unit circle into two segments of length at least $\dfrac{\pi}{k-j+1}$. 
Moreover, if the Maslov-type index $j>1$, this family will fill an open neighborhood of the origin in $\cplx^2$.

\item[(iii)] If Property ${\sf (**)}$ is not satisfied by the curve $\mathscr{C}$, then for evey $r>0$, 
the polynomial hull of $S\cap\ba{B(0;r)}$ contains at least a one parameter family 
of analytic discs
with boundary in $S$ passing through the origin.
\end{itemize}
\end{result}

\begin{result}[Wiegerinck]\label{res-wiegnonhomo}
Let $f(z,\ba{z})=p_k(z,\ba{z})+O(|z|^{k+1})$ be a smooth function of 
class $\smoo^k$ on a disc $D(0;r)$, where $p_k$ is a homogeneous polynomial of degree $k$ in $z$ and $\ba{z}$.
 Suppose that the origin is an isolated 
CR singularity of $S=\{(z,w)\in\cplx^2: w=p_k(z,\ba{z})\}$ and that the Maslov-type index at $0\in\cplx^2$ 
is $j$, $0<j<k$. Assume that $\rl(p_k(z,\ba{z})/z^{j-1})$ is subharmonic but nowhere harmonic on 
$\cplx\setminus\{0\}$ and
the curve $\mathscr {C}: S^1 \to\cplx$ defined by
\[
\mathscr{C}(z)=\dfrac{p_k(z,\ba{z})}{z^k}
\]
has the Property ${\sf (**)}$.
Then, for every $r>0$, the polynomial hull of $S\cap\ba{B(0;r)}$ contains a $(2j-1)$-parameter family of analytic discs with boundary in $S$, whose union contains an open ball centred the origin in $\cplx^2$ if $j>1$.
\end{result}


\section{The union of three totally-real planes}\label{sec-unionplanes}

The next few theorems describe a class of triples of totally-real planes in $\cplx^2$ 
whose union is locally polynomially convex at the origin, but the pair of $2\times 2$ 
matrices corresponding to these triples does not fall 
in to the open set described by Result~\ref{res-3planes}. We will use these theorems while 
proving Theorem~\ref{thm-cubic-pcvx}. These theorems might also be of independent interest.
In the statements of the next three theorems, we will assume that the planes are in Weinstock's normal form.

\begin{theorem}\label{thm-3planesspl}
Let $P_0, P_1,P_2$ be three totally-real planes in $\cplx^2$ such that
\[
P_0=\rea^2,\; \text{and}\;
P_j=(A_j+i\id)\rea^2, \; j=1,2.
\]
Assume that the pairwise unions of $P_0, P_1$, and $P_2$ are locally polynomially 
convex at $0\in\cplx^2$.
Assume further that $A_1$ and $A_2$ satisfy the following conditions:
\begin{itemize}
\item[(i)] $\det A_1=0$ and $\det A_2\geq 0$,
\item[(ii)] $\det[A_1,A_2]>0$.
\end{itemize}
Then $P_0\cup P_1\cup P_2$ is locally polynomially convex at the origin.
\end{theorem}

\begin{theorem}\label{thm-planesdetlessone}
Let $P_0, P_1, P_2$ be totally-real planes such that
\[
P_0 =\rea^2,\;
P_j =(A_j+i\id)\rea^2,\; j=1,2.
\]
Assume that the pairwise unions of $P_0, P_1$, and $P_2$ are locally polynomially 
convex at $0\in\cplx^2$. Further assume that
\begin{itemize}
\item[(i)] $|\det A_j|\leq1$ and $\det A_j<0$ for $j=1,2$,
\item[(ii)] $\det[A_1,A_2]>0$.
\end{itemize}
Then $P_0\cup P_1\cup P_2$ is locally polynomially convex at the origin.
\end{theorem}

\begin{theorem}\label{thm-planeszw}
Let $P_0, P_1, P_2$ be three totally-real planes such that 
\[
P_0 =\rea^2,\;
P_j =(A_j+i\id)\rea^2,\; j=1,2.
\]
Assume that the pairwise unions of $P_0, P_1$, and $P_2$ are locally polynomially 
convex at $0\in\cplx^2$. 
Assume further that 
\begin{itemize}
\item[(i)] $\det[A_1,A_2]>0$,
\item[(ii)] $\det A_j<0,\; j=1,2$,
\item[(iii)]$\beta(A_1, A_2)>\min\{\det A_2(\Tr A_1)^2, \det A_1 (\Tr A_2)^2\},$
\end{itemize}
where $\beta(A_1,A_2):=4\det A_1A_2-\Tr(A_1A_2)(\Tr(A_1A_2)-\Tr A_1 \Tr A_2)-\dfrac{1}{4}(\Tr A_1 \Tr A_2)^2$.
Then $P_0\cup P_1\cup P_2$ is locally polynomially convex at the origin.
\end{theorem}

\begin{proof}[Proof of Theorem~\ref{thm-3planesspl}] 
Under the conditions $i)$ and $ii)$, thanks to Lemma~\ref{lem-normal}, without 
loss of generality, we may take the matrices as:
\[
A_1=\begin{pmatrix}
0& 0\\
0 & \lambda
\end{pmatrix}
\quad\text{and}\quad
A_2=\begin{pmatrix}
s_1 & q\\
q & s_2
\end{pmatrix},
\]
where $\lambda, s_1, s_2, q\in\rea$.
The planes are of the form
\begin{align*}
P_0&=\{(x,y)\in\cplx^2: x,y\in\rea\},\\
P_1&=\{(ix, (\lm+i)y)\in\cplx^2: x,y\in\rea\},\\
P_2&=\{(s_1+i)x+qy, qx+(s_2+i)y): x,y\in\rea\}.
\end{align*}
Consider $K_j=P_j\cap \ba{\mathbb{B}(0;1)}$, $j=0,1,2$. 
We consider the polynomial 
\[
p(z,w)=z^2+w^2.
\]

Clearly,  
\[
p(K_0)\subset\{z\in\cplx: \imag z=0, \rl z\geq 0\}.
\] 
\smallskip

For $(z,w)\in K_1$, 
\[
p(z,w)=p(ix,(\lambda+i)y)=-x^2+(\lambda^2-1)y^2+2i\lambda y^2.
\]
We wish to look at the image of $K_1$ under $p$ more carefully and prove that 
$\hull{p(K_1)}\cap \hull{p(K_0)}=\{0\}$, 
and $p^{-1}\{0\}\cap K_1=\{0\}$. Our argument goes as follows:

\noindent {\sf Case I: $\lambda>0$.}
For $(z,w)\in K_1$, $p(z,w)\in \{z\in\cplx:\imag z\geq 0\}$. If $\imag p(z,w)=0$, then 
the real part $\rl p(z,w)\leq 0$, and $\rl p(z,w)=0\iff (z,w)=0$.

\noindent {\sf  Case II:  $\lambda=0$.}  In this case, $\rl p(z,w)< 0$ for all $(z,w)\in K_1\setminus \{0\}$.

\noindent {\sf Case III: $\lambda<0$.}  For $(z,w)\in K_1$, $p(z,w)\in \{z\in\cplx:\imag z\leq 0\}$, and $\imag p(z,w)=0$ implies 
the real part $\rl p(z,w)< 0$ for all $(z,w)\in K_1\setminus \{0\}$.
Therefore, in each of the above cases, 
\begin{equation}\label{eq-k0k1}
\hull{p(K_1)}\cap \hull{p(K_0)}=\{0\}, \;\;
\text{and}\;\; p^{-1}\{0\}\cap K_1=\{0\}.
\end{equation}
\smallskip

 For $(z,w)\in K_2$, 
\begin{align*}
p(z,w)&=p((s_1+i)x+qy, qx+(s_2+i)y)\\
&= (s_1^2+q^2-1)x^2+(s_2^2+q^2-1)y^2+ 
2(s_1+s_2)qxy+2i(s_1x^2+s_2y^2+2qxy).
\end{align*}
We first assume that $\det A_2>0$. This gives us $s_1s_2>0$. If $s_1>0$, $\imag p(z,w)>0$; 
if $s_1<0$, then  $\imag p(z,w)<0$ 
for all $(z,w)\in K_2\setminus \{0\}$. Therefore, $\imag p(z,w)\neq 0\; \forall (z,w)\in K_2\setminus \{0\}$.
\smallskip

We now assume that $\det A_2=0$, i. e., $q^2=s_1s_2$. Hence, $s_1$ and $s_2$ have the same sign. 
If $s_1>0$, then  
we have,  for $(z,w)\in K_2$,
\[
\imag p(z,w)=2 (\sqrt{s_1}x+\sqrt{s_2}y)^2 \geq 0.
\]
If $s_1<0$, then 
\[
\imag p(z,w)=2 (\sqrt{-s_1}x+\sqrt{-s_2}y)^2 \geq 0\quad\forall (z,w)\in K_2. 
\]
We also have
\[
\imag p(z,w)=0\impl \rl p(z,w) \leq 0.
\]
Therefore, in both the cases, we obtain that 
\begin{equation}\label{eq-k0k2}
\hull{p(K_2)}\cap \hull{p(K_0)}=\{0\}\;\;\text{and}\;\; p^{-1}\{0\}\cap K_2=\{0\}.
\end{equation}
By assumption, $K:=K_1\cup K_2$ is polynomially convex. We  now obtain, from \eqref{eq-k0k1} 
and \eqref{eq-k0k2}, that 
\[
\hull{p(K)}\cap \hull{p(K_0)}=\{0\} \;\;\text{and}\;\; p^{-1}\{0\}\cap (K\cup K_0)=\{0\}.
\]
Hence, by Result~\ref{lem-Kallin}, $K\cup K_0$ is polynomially convex. 
Therefore, $P_0\cup P_1\cup P_2$ is locally polynomially convex at the origin.
\end{proof}

\begin{proof}[Proof of Theorem~\ref{thm-planesdetlessone}]
Thanks to Lemma~\ref{lem-normal}, without loss of generality, we may assume that
\[
A_1=\begin{pmatrix}
\lambda_1& 0\\
0 & \lambda_2
\end{pmatrix}
\quad\text{and}\quad
A_2=\begin{pmatrix}
s_1 & q\\
q  & s_2
\end{pmatrix},
\]
where $\lm_j, s_j, q\in\rea$, $j=1,2.$
The planes become
\begin{align*}
P_0&=\{(x,y)\in\cplx^2: x,y\in\rea\},\\
P_1 &=\left\{ ((\lambda_1+i)x, (\lambda_2+i)y)\in\cplx^2: x,y\in \rea\right\},\\
P_2&=\left\{((s_1+i)x+qy, qx+(s_2+i)y)\in\cplx^2: x,y\in\rea\right\}.
\end{align*}
Let $K_j:=P_j\cap \ba{\ball}$, $j=0,1,2$, and let 
$K:=K_1\cup K_2$. 
Consider the polynomial 
\[
p(z,w)=z^2+w^2.
\]
We have 
$
p(K_0)\subset\{z\in\cplx: z\geq 0\}
$
and $p^{-1}\{0\}\cap K_0=\{0\}$. We now divide the remaining proof into three cases.
\smallskip

{\em {\bf Case I.} $|\det A_j|<1$ for $j=1,2$.}

We now look carefully at the image of $K_1$ under the polynomial $p$. 
\begin{equation}\label{eq-imagek1underp}
p((\lm_1+i)x, (\lm_2+i)y)= (\lm_1^2-1)x^2+(\lm_2^2-1)y^2+2i(\lm_1x^2+\lm_2y^2).
\end{equation}
The imaginary part  
\begin{equation}\label{eq-imaginaryimagek1}
\imag p((\lm_1+i)x, (\lm_2+i)y)=2(\lm_1x^2+\lm_2y^2).
\end{equation}
 Since 
$\det A_1=\lm_1\lm_2<0$, one of $\lm_1$ or $\lm_2$ must be positive and the other is negative.
Assuming $\lm_2< 0$, we see that
\begin{equation}\label{eq-imaginaryp1}
\imag p((\lm_1+i)x, (\lm_2+i)y) =0 \iff y^2 = -\dfrac{\lm_1}{\lm_2} x^2.
\end{equation}
Then the real part of $p((\lm_1+i)x, (\lm_2+i)y)$ becomes, in view of \eqref{eq-imaginaryp1},
\begin{equation}\label{eq-realimagek1}
\rl p((\lm_1+i)x, (\lm_2+i)y) 
=\dfrac{1}{\lm_2}(\lm_1-\lm_2)(\lm_1\lm_2+1)x^2  \leq 0.
\end{equation}
Since $|\det A_1|<1$, we obtain that
$\rl p((\lm_1+i)x, (\lm_2+i)y)=0\iff  (x,y)=0$.
Thus, we obtain that $\hull{p(K_0)}\cap \hull{p(K_1)}=\{0\}$ and $p^{-1}\{0\}\cap K_1=\{0\}$.
\smallskip

For $(z,w)\in K_2$
\begin{equation}\label{eq-imagek2underp}
p(z,w)= (s_1^2+q^2-1)x^2+(s_2^2+q^2-1)y^2 +2q(s_1+s_2)xy +2i(s_1x^2+2qxy+s_2y^2).
\end{equation}
The imaginary part 
\begin{equation}\label{eq-imaginaryimagek2}
\imag p((s_1+i)x+qy, qx+(s_2+i)y)=2(s_1x^2+2qxy+s_2y^2).
\end{equation}
We have 
\begin{equation}\label{eq-imaginaryp2}
\imag p((s_1+i)x+qy, qx+(s_2+i)y)=0 \iff 2qxy= -(s_1x^2+s_2y^2).
\end{equation}
The real part becomes, in view of \eqref{eq-imaginaryp2},
\begin{align}
\rl p((s_1+i)x+qy, qx+(s_2+i)y) 
&= (q^2-s_1s_2-1)(x^2+y^2)\notag\\
&=-(\det A_2 +1)(x^2+y^2).\label{eq-realimagek2}
\end{align}
Since $|\det A_2|<1$, $p(K_0)\cap p(K_2)=\{0\}$.
Therefore, 
\[
\hull{p(K_0)}\cap \hull{p(K)}=\{0\}\;\text{and}\;\; p^{-1}\{0\}\cap (K_0\cup K)=\{0\}. 
\]
Hence, by Result~\ref{lem-Kallin}, $K_0\cup K$ is polynomially convex.
\smallskip

The next two cases fall under the condition $|\det A_j|=1$ for some $j=1,2.$ Since $\det A_j<0$, we only need to consider $\det A_j=-1$ for some $j=1,2$. We now consider the other two cases.
\smallskip

{\em {\bf Case II.} $\det A_j=-1$ for some $j\in\{1,2\}$ and $|\det A_i|<1$ for $i\neq j$.}

We again consider the Kallin's polynomial $p(z,w)=z^2+w^2$. We first look at the image of $K_1$ under the polynomial p. Equations \eqref{eq-imagek1underp}, \eqref{eq-imaginaryimagek1}, \eqref{eq-imaginaryp1}, and \eqref{eq-realimagek1} remain the same. Clearly, $p(K_0)\cap p(K_1)=\{0\}$. If $|\det A_1|<1$ and  \eqref{eq-imaginaryp1} holds, then, by \eqref{eq-realimagek1}, we get that
\[
\rl p((\lm_1+i)x, (\lm_2+i)y)=0\iff  (x,y)=0.
\]
Therefore, 
\begin{equation}\label{eq-pinverse0capk1a}
p^{-1}\{0\}\cap K_1=\{0\}.
\end{equation}
If $\det A_1=-1$ and  \eqref{eq-imaginaryp1} holds, then, by \eqref{eq-realimagek1}, we obtain that 
\[
\rl p((\lm_1+i)x, (\lm_2+i)y)\equiv 0.
\]
This implies that 
\begin{equation}\label{eq-pinverse0capk1b}
p^{-1}\{0\}\cap K_1=\{((\lm_1+i)x,(\lm_2+i)y)\in \cplx^2: y^2=-\dfrac{\lm_1}{\lm_2}x^2\},
\end{equation}
which is the union of two real line segments in $\cplx^2$ intersecting only at the origin.

We now look at the image of $K_2$. Equations \eqref{eq-imagek2underp}, \eqref{eq-imaginaryimagek2}, \eqref{eq-imaginaryp2}, and \eqref{eq-realimagek2} remain the same. Clearly, $p(K_0)\cap p(K_2)=\{0\}$.
If $|\det A_2|<1$ and  \eqref{eq-imaginaryp2} holds, then, by \eqref{eq-realimagek2}, we get that
\[
\rl p((s_1+i)x+qy, qx+(s_2+i)y) =0 \iff (x,y)=0.
\]
Therefore, 
\begin{equation}\label{eq-pinverse0capk2a}
p^{-1}\{0\}\cap K_2=\{0\}.
\end{equation}
If $\det A_2=-1$ and  \eqref{eq-imaginaryp2} holds, then, by \eqref{eq-realimagek2}, we obtain that 
\[
\rl p((s_1+i)x+qy, qx+(s_2+i)y) \equiv 0.
\]
This gives that $p^{-1}\{0\}\cap K_2$ is the union of two real line segments in $\cplx^2$ intersecting only at the origin. 
Therefore, in this case, we have $\hull{p(K_0)}\cap \hull{P(K)}=\{0\}$ and $p^{-1}\{0\}\cap (K_0\cup K)$ is the union of two real line segments in $\cplx^2$ intersecting only at the origin. Hence, $p^{-1}\{0\}\cap (K_0\cup K)$ is polynomially convex.
\smallskip

{\em {\bf Case III.} $\det A_j=-1$ for $j=1,2$.}

We consider the same Kallin's polynomial $p(z,w)=z^2+w^2$. 
We now need to consider carefully $p^{-1}\{0\}\cap K_j$ for $j=1,2$. 
If $\det A_1=-1$, then we see, from the above computations, that, for each $z\in K_1$, the real part $\rl p(z)$ vanishes whenever $\imag p(z)=0$. The similar is true when $\det A_2=-1$.
We also see that, if $\det A_j=-1$, then the set $p^{-1}\{0\}\cap K_j=\{z\in K_j: \imag p(z)=0\}$ is a union of two real line segments in $\cplx^2$ intersecting at 
the origin. Therefore, $p^{-1}\{0\}\cap K$ is, a union of four real line 
segments in $\cplx^2$ intersecting pairwise only at the origin, which is polynomially convex. Hence, by Kallin's lemma, $K_0\cup K$ is polynomially convex.
\end{proof}

\begin{proof}[Proof of Theorem~\ref{thm-planeszw}]
Since $A_1,A_2\in\rea^{2\times 2}$ and $\det A_j<0$ for $j=1,2$, both the matrices have real eigenvalues. 
Assume, without of generality (otherwise we will interchange the matrices), that
\begin{equation}\label{eq-assump}
\beta(A_1,A_2)>\det A_2(TrA_1)^2.
\end{equation}
Since $\det[A_1,A_2]>0$, by applying Lemma~\ref{lem-normal}, we get, up to a 
simultaneous similarity by real matrix, that
\[
A_1=\begin{pmatrix}
\lambda_1 & 0\\
0 & \lambda_2
\end{pmatrix}\;\;\text{and}\;\;
A_2=\begin{pmatrix}
s_1 & q\\
q & s_2
\end{pmatrix}.
\]	
We first view \eqref{eq-assump} in terms of the entries of the matrices $A_1$ and $A_2$. 
\[
A_1A_2=\begin{pmatrix}
      \lm_1s_1 & \lm_1q\\
      \lm_2q & \lm_2s_2.
      \end{pmatrix}
      \]
 Therefore, we compute:
 \begin{align*}
 \beta(A_1,A_2)
 &=4\lm_1\lm_2(s_1s_2-q^2)+(\lm_1s_1+\lm_2s_2)(\lm_1s_2+\lm_2s_1)-\dfrac{1}{4}(\lm_1+\lm_2)^2(s_1+s_2)^2\\
 &=4\lm_1\lm_2(s_1s_2-q^2)-\dfrac{1}{4}(\lm_1-\lm_2)^2(s_1-s_2)^2. 
 \end{align*}
Hence, 
\begin{equation}\label{eq-assump2}
\beta(A_1,A_2)>\det A_2(TrA_1)^2 \iff q^2>\dfrac{(s_1+s_2)^2}{4}.
\end{equation}
We now consider compact sets
\[
 K_j=P_j\cap\ba{\ball},\;j=0,1,2,
\]
and the polynomial 
\[
p(z,w)=zw.
\]
Clearly, 
\begin{equation}\label{eq-imagek0}
p(K_0)\subset\rea.
\end{equation}
 Since
$p((\lm_1+i)x, (\lm_2+i)y)= (\lm_1\lm_2-1)xy+i(\lm_1+\lm_2)xy,$
hence, we obtain that 
\begin{equation}\label{eq-imagek1}
p(K_1)\subset\left\{u+iv\in\cplx: (\lm_1+\lm_2)u=(\lm_1\lm_2-1)v\right\}.
\end{equation}
We also get that
\begin{align*}
p((s_1+i)x+qy, qx+(s_2+i)y)&= ((s_1+i)x+qy)(qx+(s_2+i)y)\\
&=(s_1x+qy)(qx+s_2y)-xy+i(qx^2+(s_1+s_2)xy+qy^2).	
\end{align*}
The imaginary part of $p((s_1+i)x+qy, qx+(s_2+i)y)$ is 
$qx^2+(s_1+s_2)xy+qy^2$, which vanishes, thanks to \eqref{eq-assump2},  if and only if $(x,y)=(0,0)$.
Therefore, we have 
\begin{equation}\label{eq-imagek2}
p(K_2)\subset\{u+iv\in\cplx:\; v\ne 0\}\cup \{0\}.	
\end{equation}
We now divide the remaining part of the proof into two cases.
\smallskip

{\bf Case I.} $\Tr A_1\neq 0$.

In this case, we consider $K:=K_1\cup K_2$.
From \eqref{eq-imagek0},\eqref{eq-imagek1} and \eqref{eq-imagek2}, we obtain that $\hull{p(K_0)}\cap\hull{p(K)}\subset\{0\}$. We also see that  
$p^{-1}\{0\}\cap (K_0\cup K)$ is a union of two real line segments intersecting only at the origin, which is polynomially convex. Hence, by Kallin's lemma (Result~\ref{lem-Kallin}), 
$K_0\cup K=K_0\cup K_1\cup K_2$ is polynomially convex.
\smallskip

{\bf Case II.} $\Tr A_1= 0$.

In this case, we consider $K:=K_0\cup K_1$. Since $\Tr A_1=0$, from Equation~\eqref{eq-imagek1}, we have 
\[
p(K_1)\subset \rea.
\]
Therefore, $p(K)\subset \rea$. In this case also Equation~\eqref{eq-imagek2} holds. Therefore, 
We obtain $\hull{p(K)}\cap \hull{p(K_2)}=\{0\}$. Again, as before, $p^{-1}\{0\}\cap (K\cup K_2)$ is a union of two real line segments in $\cplx^2$ intersecting only at the origin. Hence, $p^{-1}\{0\}\cap (K\cup K_2)$ is polynomially convex. Therefore, by Kallin's lemma (Result~\ref{lem-Kallin}) $K\cup K_2$ is polynomially convex.
\end{proof}

\section{The union of three totally-real surfaces}\label{sec-polyperturb}

The next couple of theorems are the most important theorems of this section. Those will be crucially used in our proof of Theorem~\ref{thm-surface-pcvx}.
\begin{theorem}\label{thm-unionsurfaces}
Let $S_0, S_1$ and $S_2$ be three totally-real surfaces in $\cplx^2$ 
intersecting transversely at the origin such that 
\begin{align*}
T_0S_0 &=\rea^2\\
T_0S_1 &= (A_1+i\id)\rea^2\\
T_0S_2 &= (A_2+i\id)\rea^2,
\end{align*}
where $A_1,A_2\in M_2(\rea)$. Assume that the pairwise unions of $T_0S_0, T_0S_1$ and $T_0S_2$ is locally polynomially convex at the origin.
Assume further that $det[A_1,A_2]>0$. Then  $S_0\cup S_1\cup S_2$ is locally polynomially 
convex at the origin if one of the following holds:
\begin{itemize}
\item[(i)] $\det A_j>0$ for $j=1,2$; 
\item[(ii)] $\det A_j<0$ and $|\det A_j|< 1$
\end{itemize}

\end{theorem}

\begin{proof}
Since $A_1,A_2\in M_2(\rea)$ and $det[A_1,A_2]>0$, we obtain, in view of Lemma~\ref{lem-normal}, that 
there exists 
an invertible matrix $T\in\rea^{2\times 2}$ 
such that 
\[
TA_1T^{-1}=\begin{pmatrix}
\lambda_1& 0\\
0 & \lambda_2
\end{pmatrix}
\quad\text{and}\quad
TA_2T^{-1}=\begin{pmatrix}
s_1 & q\\
q & s_2
\end{pmatrix}
\]
for $\lambda, s_j,t_j\in\rea$, $j=1,2.$
Therefore, without loss of generality, we may assume that 
\[
A_1=\begin{pmatrix}
\lambda_1& 0\\
0 & \lambda_2
\end{pmatrix}
\quad\text{and}\quad
A_2=\begin{pmatrix}
s_1 & q\\
q & s_2
\end{pmatrix}.
\]
Since $S_0$ is a totally-real surface passing through the origin and $T_0S_0=\rea^2\subset\cplx^2$, 
by implicit mapping theorem there exists a $\delta>0$ such that 
\[
S_0\cap B(0;\delta)=\{(x+if_0(x,y), y+ig_0(x,y))\in\cplx^2:\;\text{for small enough}\; x,y\in\rea\},
\] 
where $f_0$ and $g_0$ are real valued functions.
Locally around the origin we can view the totally-real surfaces as follows: 
\begin{align}
S_0(\delta) &:=\{(x+if_0(x,y), y+ig_0(x,y))\in\cplx^2: (x,y)\in\rea^2\cap \ba{B(0;\delta)}\}\\
S_1(\delta) &:=\{((\lambda_1+i)x+f_1(x,y), (\lambda_2+i)y+g_1(x,y))\in\cplx^2: (x,y)\in\rea^2\cap \ba{B(0;\delta)}\}\\
S_2(\delta) &:=\{((s_1+i)x+qy+f_2(x,y), qx+(s_2+i)y+g_2(x,y))\in\cplx^2: (x,y)\in\rea^2\cap \ba{B(0;\delta)}\},
\end{align}
where $f_0, g_0$ are real valued, and $f_j(x,y)\sim o(|(x,y)|)$ and $g_j(x,y)\sim o(|(x,y)|)$ for all $j=0,1,2$. 
\medskip

We first assume $(i)$, i.e., $det A_j>0$ for $j=1,2$. We will use Kallin's lemma to conclude the 
local polynomial convexity of $S_0(\delta)\cup S_1(\delta)\cup S_2(\delta)$ for some small enough $\delta>0$.
Consider the polynomial 
\[
p(z,w)=z^2+w^2.
\]
Clearly, 
\begin{equation}\label{eq-images0}
p(x+if_0(x,y), y+ig_0(x,y)) = x^2+y^2 +o(|(x,y)|^2).
\end{equation}
We also have 
\begin{align}
& p((\lambda_1+i)x+f_1(x,y), (\lambda_2+i)y+g_1(x,y))\\ &= (\lambda_1+i)^2 x^2+(\lambda_2+i)^2 y^2 +o((|(x,y)|^2) \notag\\
&= (\lambda_1^2-1)x^2+(\lambda_2^2-1)y^2+2i(\lambda_1 x^2+\lambda_2 y^2)+o((|(x,y)|^2) \notag.
\end{align}
Since $detA_1>0$, the imaginary part of $p((\lambda_1+i)x+f_1(x,y), (\lambda_2+i)y+g_1(x,y))$ is
$2(\lambda_1 x^2+\lambda_2 y^2)+o((|(x,y)|^2)\geq0$ or $\leq 0$ for $(x,y)\in\rea^2\cap B(0;\delta)$ and equal to zero only when 
$(x,y)=(0,0)$. Therefore there exists a constant $C_1>0$ such that 
\begin{equation}\label{eq-bounds1}  
p(S_1(\delta))\subset \{u+iv\in\cplx: |u|<C_1|v|\}.
\end{equation} 
We also compute: 
\begin{align}
& p((s_1+i)x+qy+f_2(x,y), qx+(s_2+i)y+g_2(x,y)) \notag\\
&= (s_1^2+q^2-1)x^2+(s_2^2+q^2-1)y^2+ 2i(s_1x^2+s_2y^2+2qxy)+o(|(x,y)|^2)\notag.
\end{align}
Since $det A_2>0$, the imaginary part of the above equation 
$2(s_1x^2+s_2y^2+2qxy)+o(|(x,y)|^2)\geq 0$ or $\leq 0$ and equal to zero only when 
$(x,y)=(0,0)$. Therefore, there exists a constant $C_2>0$ such that 
\begin{equation}\label{eq-bounds2}
p(S_2(\delta))\subset \{u+iv\in\cplx: |u|<C_2|v|\}.
\end{equation}
Consider $C:=max\{C_1,C_2\}$, and choose $\eps>0$ such that $\eps<1/C$. 
From Equation~\eqref{eq-images0} we see that the imaginary part of 
$p(x+if_0(x,y), y+ig_0(x,y))$ is $o(|z|^2)$. Therefore, shrinking $\delta$ if necessary, we 
obtain 
\begin{equation}\label{eq-bounds0}
p(S_0(\delta))\subset \{u+iv\in\cplx: |v|<\eps|u|\}.
\end{equation}
From Equations~\eqref{eq-bounds1}, \eqref{eq-bounds2} and \eqref{eq-bounds0} we have 
\[
\hull{p(S_0(\delta)}\cap(p(S_1(\delta ) \cup S_2(\delta) \hull{)}=\{0\}.
\]
We also have from above computations that 
\[
p^{-1}\{0\}\cap (S_0(\delta)\cup S_1(\delta) \cup S_2(\delta))=\{0\}.
\]
Hence, all the conditions of Kallin's lemma holds and we conclude 
$S_0\cup S_1\cup S_2$ is locally polynomially convex at the origin.
\medskip

Finally, we assume $(ii)$, i.e., $\det A_j<0$ and $|\det A_j|<1$. We consider the matrices in the normal form 
\[
A_1=\begin{pmatrix}
\lm_1 & 0\\
0 & \lm_2
\end{pmatrix}
\quad\text{and}\quad
A_2=\begin{pmatrix}
s_{1} & q\\
q & s_{2}
\end{pmatrix}.
\]
Consider the polynomial 
\[
p(z,w)=z^2+w^2.
\]
Let $D_\dl=\rea^2\cap \ba{B(0;\dl)}$ as before.
We have 
\[
p(x, y)=x^2+y^2.
\]
Hence,
$p^{-1}\{0\}\cap S_0(\dl)=\{0\}$. In fact, for given $\eps>0$ there exists a $\dl>0$ such that 
\begin{equation}\label{eq-imageS0}
    p(S_0(\dl))\subset\{w\in\cplx: \rl w\geq 0,\; |\imag w|\leq \eps \rl w\}.
\end{equation}

By using the computations in the proof of Theorem~\ref{thm-planesdetlessone} we get that 
\begin{equation}\label{eq-surfimagek1underp}
p((\lm_1+i)x+f_1(x,y), (\lm_2+i)y+g_1(x,y))= (\lm_1^2-1)x^2+(\lm_2^2-1)y^2+2i(\lm_1x^2+\lm_2y^2)+o(|(x,y)|^2).
\end{equation}
Equation \eqref{eq-surfimagek1underp} gives us that the imaginary part  
\begin{equation}\label{eq-surfimaginaryimagek1}
\imag p((\lm_1+i)x+f_1(x,y), (\lm_2+i)y+f_2(x,y))=2(\lm_1x^2+\lm_2y^2)+o(|(x,y)|^2).
\end{equation}
From \eqref{eq-surfimaginaryimagek1} we see that
$\imag p((\lm_1+i)x+f_1(x,y), (\lm_2+i)y+g_1(x,y)) =0$ in a curve which is a small perturbation of the curve $y^2 = -\dfrac{\lm_1}{\lm_2} x^2$. Then the real part of $p((\lm_1+i)x+f_1(x,y), (\lm_2+i)y+g_1(x,y))$ becomes, shrinking $\dl$ if required,
\begin{equation}\label{eq-surfrealimagek1}
\rl p((\lm_1+i)x+f_1(x,y), (\lm_2+i)y+g_1(x,y)) 
=\dfrac{1}{\lm_2}(\lm_1-\lm_2)(\lm_1\lm_2+1)x^2 +o(|(x,y)|^2)  \leq 0.
\end{equation}
Since $|\det A_1|<1$, we obtain that
$\rl p((\lm_1+i)x+f_1(x,y), (\lm_2+i)y+f_2(x,y))=0\iff  (x,y)=0$.
Hence, there is a conical neighbourhood $V$ of $\{(z,w)\in S_1(\dl): \imag p(z,w) =0\}$ such that  $\rl p(z,w)\leq 0$ for all $(z,w)\in V$.
Thus, from \eqref{eq-imageS0} and \eqref{eq-surfrealimagek1}, we obtain, shrinking $\dl>0$ if necessary, that $\hull{p(S_o(\dl)}\cap \hull{p(S_1(\dl))}=\{0\}$ and $p^{-1}\{0\}\cap S_1(\dl)=\{0\}$.
\smallskip

For $(z,w)\in S_2(\dl)$
\[
p(z,w)= (s_1^2+q^2-1)x^2+(s_2^2+q^2-1)y^2 +2q(s_1+s_2)xy +2i(s_1x^2+2qxy+s_2y^2)+o(|(x,y)|^2).
\]
The imaginary part 
\[
\imag p(z,w)=2(s_1x^2+2qxy+s_2y^2)+o(|(x,y)|^2).
\]
We have the imaginary part vanishes in a curve which is a small perturbation of the curve
\[
2qxy= -(s_1x^2+s_2y^2).
\]
Then, in view of \eqref{eq-imaginaryp2}, the real part becomes
\begin{align}
& \rl p((s_1+i)x+qy+f_2(x,y), qx+(s_2+i)y+g_2(x,y)) \notag\\
&= (q^2-s_1s_2-1)(x^2+y^2)+o(|(x,y)|^2)\notag\\
&=-(\det A_2 +1)(x^2+y^2)+o(|(x,y)|^2)\notag\\
& < 0\;\text{and}\; =0\;\text{if and only if}\; (x,y)=0. \notag
\end{align}
Hence, there is a conical neighbourhood $W$ of $\{(z,w)\in S_2(\dl): \imag p(z,w) =0\}$ such that  $\rl p(z,w)\leq 0$ for all $(z,w)\in V$.

Therefore, shrinking $\dl>0$ further, if required, we obtain that
\[
\hull{p(S_0(\dl)}\cap \hull{p(S_2(\dl))}=\{0\}\;\text{and}\;\; p^{-1}\{0\}\cap (S_0(\dl) \cup S_2(\dl))=\{0\}. 
\]
Hence, by Result~\ref{lem-Kallin}, $S_0(\dl)\cup S_1(\dl)\cup S_2(\dl)$ is polynomially convex.
Therefore, $S_0\cup S_1\cup S_2$ is locally polynomially convex at $0\in\cplx^2$
\end{proof}

\begin{theorem}\label{thm-unionsurfacesdet0}
Let $S_0, S_1$ and $S_2$ be three totally-real surfaces in $\cplx^2$ 
intersecting transversely at the origin such that 
$T_0S_0 =\rea^2, \;
T_0S_1 = (A_1+i\id)\rea^2$ and
$T_0S_2 = (A_2+i\id)\rea^2$. Assume that the 
pairwise unions of $T_0S_0, T_0S_1$ and $T_0S_2$ are locally 
polynomially convex at the orgin.
Assume further that
\begin{itemize}
\item[i)] $\det A_1=0$ and $\det A_2\geq 0$,
\item[ii)] $\det[A_1,A_2]>0$.
\end{itemize}
Then $S_0\cup S_1\cup S_2 $ is locally polynomially convex at the origin.
\end{theorem}

\begin{proof}
Under the conditions $i)$ and $ii)$, by Lemma~\ref{lem-normal}, without 
loss of generality we assume that
\[
A_1=\begin{pmatrix}
0& 0\\
0 & \lambda
\end{pmatrix}
\quad\text{and}\quad
A_2=\begin{pmatrix}
s_{21} & t_2\\
t_2 & s_{22}
\end{pmatrix}.
\]

Locally around the origin we can view the totally-real surfaces as follows: 
\begin{align*}
S_0(\delta) &:=\{(x+if_0(x,y), y+ig_0(x,y))\in\cplx^2: (x,y)\in\rea^2\cap \ba{B(0;\delta)}\}\\
S_1(\delta) &:=\{(ix+f_1(x,y), (\lambda+i)y+g_1(x,y))\in\cplx^2: (x,y)\in\rea^2\cap \ba{B(0;\delta)}\}\\
S_2(\delta) &:= \{((s_{21}+i)x+t_2y+f_2(x,y), t_2x+(s_{22}+i)y+g_2(x,y))\in\cplx^2:
(x,y)\in\rea^2\cap \ba{B(0;\delta)}\}
\end{align*}
where $f_0, g_0$ are real valued, and $f_j(x,y)\sim o(|(x,y)|)$ and $g_j(x,y)\sim o(|(x,y)|)$ for all $j=0,1,2$. 
We now divide our arguments in two cases: 
\smallskip

\noindent {\bf  Case I.} Assume $\det A_1=0$ and $\det A_2=0$.

We consider the polynomial 
\[
p(z,w)=z^2+w^2.
\]
Clearly,  
\begin{equation}\label{eq-imager2det0}
p(x+if_0(x,y), y+ig_0(x,y)) = x^2+y^2 +o(|(x,y)|^2).
\end{equation}
Hence, $\rl{p(x+if_0(x,y), y+ig_0(x,y))}\geq 0$,  and $=0$ if and only if $(x,y)=0$. 
\smallskip

For $(z,w)\in S_1(\delta)$, 
\[
p(z,w)=p(ix+f_1(x,y),(\lambda+i)y+g_1(x,y))=-x^2+(\lambda^2-1)y^2+2i\lambda y^2+o(|(x,y)|^2).
\]
We first note that on the set $S_1(\delta)\cap \{y=0\}$ the real part $\rl p\leq 0$ and equal to zero if and only if $x=0$. 
 Therefore, by continuity, there exists $\eps>0$ such that $\rl p <0$ on 
 $\{(ix+f_1(x,y), (\lambda+i)y+g_1(x,y))\in\cplx^2: (x,y)\in\rea^2\cap \ba{B(0;\delta)}, |y|<\eps|x|\}$. 
 On the other part of $S_1(\delta)$, we get the the imaginary part $\imag p\geq 0$ and $=0$ only when $(x,y)=0$. 
 Hence, there exists a constant $C_1>0$ such that 
 $|\rl p(z)|\leq C_1 \imag p(z)$ 
 for all $z\in\{(ix+f_1(x,y), (\lambda+i)y+g_1(x,y))\in\cplx^2: (x,y)\in\rea^2\cap \ba{B(0;\delta)}, |y|\geq\eps|x|\}$
\smallskip

\noindent For $(z,w)\in K_2$, 
\begin{align*}
p(z,w)&=p((s_{21}+i)x+t_2y +f_2(x,y), t_2x+(s_{22}+i)y+g_2(x,y))\\
&= (s_{21}^2+t_2^2-1)x^2+(s_{22}^2+t_2^2-1)y^2+ 2(s_{21}+s_{22})t_2xy\\
&\qquad\qquad\qquad\qquad\qquad\qquad+2i(s_{21}x^2+s_{22}y^2+2t_2xy) +o(|(x,y)|^2).
\end{align*}
Since $\det A_2=0$,  
we have 
\[
\imag p(z,w)=2 (\sqrt{s_{21}}x+\sqrt{s_{22}}y)^2 +o(|(x,y)|^2).
\]
On the set $S_2(\delta)\cap \{\sqrt{s_{21}}x+\sqrt{s_{22}}y=0\}$ 
the real part 
\[
\rl p= -(1+s_{22}/s_{21})y^2+o(|y|^2)\leq 0,
\]
 and equal to zero if and only if $y=0$.
Again, by continuity of the real part, there exists $\eps_1>0$ such that 
 $\rl p <0$ on 
 $\{((s_{21}+i)x+t_2y+f_2(x,y), t_2x+(s_{22}+i)y+g_2(x,y))\in\cplx^2: (x,y)\in\rea^2\cap \ba{B(0;\delta)}, 
 |x|<(\eps_1-\sqrt{s_{22}}/\sqrt{s_{21}})|y|\}$.
 On the other part of $S_2(\delta)$, we get the the imaginary part $\imag p\geq 0$ and $=0$ only when $(x,y)=0$. 
 Hence, there exists a constant $C_2>0$ such that 
 $|\rl p(z)|\leq C_2 \imag p(z)$ for all 
 $z\in\{((s_{21}+i)x+t_2y+f_2(x,y), t_2x+(s_{22}+i)y+g_2(x,y))\in\cplx^2: (x,y)\in\rea^2\cap \ba{B(0;\delta)}, 
 |x|\geq(\eps_1-\sqrt{s_{22}}/\sqrt{s_{21}})|y|\}$. 

 From \eqref{eq-imager2det0} we get that the 
imaginary part $\imag{p(x+if_0(x,y), y+ig_0(x,y))}\sim o(|(x,y)|^2)$.
 Since  
 \[
 \rl{p(x+if_0(x,y), y+ig_0(x,y))}\geq 0,
 \]
  and $=0$ if and only if $(x,y)=0$, we get a sufficiently small $\eps_2>0$ 
 ( $\eps_2<\min {1/C_1, 1/C_2}$)
 such that 
 \[
 p(S_0(\delta))\subset \{z\in\cplx: |\imag z|\leq \eps_2 |\rl z|\}.
 \]
 Hence, all the conditions of Kallin's lemma are satisfied. Therefore, shrinking $\delta$ if necessary, we conclude 
 that $S_0(\delta)\cup S_1(\delta)\cup S_2(\delta)$ is polynomially convex.
 \smallskip
 
 \noindent {\bf Case II.} Assume $\det A_1=0$ and $\det A_2>0$. 
 \smallskip
 
 We will deal the surface $S_1(\delta)$ as in Case I and for other two surfaces we will use computations 
 done in the proof of Theorem~\ref{thm-unionsurfaces}. Then, by using Kallin's lemma, we conclude 
 that $S_0\cup S_1\cup S_2$ is locally polynomially convex.
 \end{proof}

\section{Proofs of Theorem~\ref{thm-3planes}, Theorem~\ref{thm-cubic-pcvx} and Theorem~\ref{thm-surface-pcvx}}\label{sec-proofs-pcvx}

We begin this section by proving Theorem~\ref{thm-3planes}. This theorem plays a vital role in the proof of the 
other two theorems.
\begin{proof}[Proof of Theorem~\ref{thm-3planes}]
We will work with Weinstock's normal form of the triples of totally-real planes as in \eqref{eq-planemat}
\[
P_0^t=\rea^2,\;
P_j^t=(A_j^t+i\id)\rea^2,\; j=1,2,
\]
where 
\[
A_1^t=\begin{pmatrix}
\dfrac{t}{\sqrt{3}(1+t)} & \dfrac{1}{1+t}\\
-\dfrac{1}{1-t} & -\dfrac{t}{\sqrt{3}(1-t)}
\end{pmatrix}
\quad\text{and}\quad
A_2^t=\begin{pmatrix}
-\dfrac{t}{\sqrt{3}(1+t)} & \dfrac{1}{1+t}\\
-\dfrac{1}{1-t} & \dfrac{t}{\sqrt{3}(1-t)}
\end{pmatrix}.
\]
We have
\begin{align*}
\sigma(A_1^t)&=\left\{\dfrac{-t^2+\sqrt{4t^2-3}}{\sqrt{3}(1-t^2)}, \dfrac{-t^2-\sqrt{4t^2-3}}{\sqrt{3}(1-t^2)}\right\},\\
\sigma(A_2^t)&=\left\{\dfrac{t^2+\sqrt{4t^2-3}}{\sqrt{3}(1-t^2)}, \dfrac{t^2-\sqrt{4t^2-3}}{\sqrt{3}(1-t^2)}\right\}.
\end{align*}
We note that the eigenvalues of $A_1^t$ and $A_2^t$ are not purely imaginary for $t\in (0,\infty)$. Hence, by Theorem~\ref{thm-weinstock}, 
$P_0^t\cup P_1^t$ and 
$P_0^t\cup P_2^t$ are locally polynomially convex at $0\in\cplx^2$. To show $P_1^t\cup P_2^t$ 
is polynomially convex, we need some computations. 
First, we apply an invertible $\cplx$-linear transformation $\cplx^2\longrightarrow\cplx^2$ by $z\mapsto (A_1^t-i\id)z$. It maps 
$P_1^t$ to $\rea^2$, and $P_2^t$ to $(B^t+i\id)\rea^2$, where 
$B^t=(A_1^tA_2^t+\id)(A_1^t-A_2^t)^{-1}$. It suffices to show that $Tr B^t\neq 0$. We compute: 
\begin{align}
A_1^tA_2^t
&= \begin{pmatrix}
-\dfrac{t^2}{3(1+t)^2}-\dfrac{1}{1-t^2} & \dfrac{t}{\sqrt{3}(1+t)^2}+\dfrac{t}{\sqrt{3}(1-t^2)}\\
\dfrac{t}{\sqrt{3}(1-t)^2}+\dfrac{t}{\sqrt{3}(1-t^2)} & -\dfrac{t^2}{3(1-t)^2}-\dfrac{1}{1-t^2}
\end{pmatrix},\label{eq-A1A2}\\
A_2^tA_1^t
&=\begin{pmatrix}
-\dfrac{t^2}{3(1+t)^2}-\dfrac{1}{1-t^2} & -\dfrac{t}{\sqrt{3}(1+t)^2}-\dfrac{t}{\sqrt{3}(1-t^2)}\\
-\dfrac{t}{\sqrt{3}(1-t)^2}-\dfrac{t}{\sqrt{3}(1-t^2)} & -\dfrac{t^2}{3(1-t)^2}-\dfrac{1}{1-t^2}
\end{pmatrix},\label{eq-A2A1}
\end{align}
\begin{align*}
(A_1^t-A_2^t)^{-1} &=\dfrac{\sqrt{3}(1-t^2)}{2t}\begin{pmatrix}
\dfrac{1}{1-t} & 0\\
0 & -\dfrac{1}{1+t}
\end{pmatrix},\\
(A_1^tA_2^t+\id)&=\begin{pmatrix}
1-\dfrac{3+3t+t^2-t^3}{3(1+t)(1-t^2)} & \dfrac{2t}{\sqrt{3}(1+t)(1-t^2)}\\
\dfrac{2t}{\sqrt{3}(1-t)(1-t^2)} & 1-\dfrac{3-3t+t^2+t^3}{3(1-t)(1-t^2)}
\end{pmatrix}.
\end{align*}
Thus, we obtain that
\[
Tr (B^t)=Tr\left((A_1^tA_2^t+\id)(A_1^t-A_2^t)^{-1}\right)= \sqrt{3}\left(1-\dfrac{3-t^2}{3(1-t^2)}\right)\neq 0\;\; \forall t\in (0,\infty).
\]
 Hence, the pairwise union of the totally-real planes $P_0,P_1$ and $P_2$ is polynomially convex.
 From \eqref{eq-A1A2} and \eqref{eq-A2A1}, we obtain that 
\begin{align}
\det A_1^t&=\det A_2^t=\dfrac{3-t^2}{3(1-t^2)} \label{eq-detA}\\
Tr A_1^t&=-\dfrac{2t^2}{\sqrt{3}(1-t^2)}, \quad Tr A_2^t=\dfrac{2t^2}{\sqrt{3}(1-t^2)}\label{eq-trA}\\
Tr A_1^tA_2^t&=-\dfrac{2(t^4-2t^2+3)}{3(1-t^2)^2}\label{eq-trA1A2},\\
\det[A_1^t, A_2^t]&=-\dfrac{16t^2}{3(1-t^2)^3}\label{eq-detcomm}.
\end{align}
We now divide the remaining part of the proof into four cases.
\medskip

\noindent {\bf Case I: $t^2>3$.}
\smallskip

\noindent In this case, we note that
\[
\det[A_1^t,A_2^t]>0
\;\;
\text{and}\;\;
\det A_j^t>0,\;\; j=1,2.
\]
Hence, we have $\det A_j^t\det[A_1^t,A_2^t]>0$ for $j=1,2$. Therefore, in view of Part $(i)$ of 
Theorem~\ref{res-3planes}, we get that $P_0^t\cup P_1^t\cup P_2^t$ is locally polynomially 
convex at the origin.
\smallskip

\noindent {\bf Case II: $t^2=3$}
\smallskip

\noindent We note that
\[
\det[A_1^t,A_2^t]>0\;\;\text{and}\;\;\det A_j^t=0,\quad j=1,2,
\]
Hence, by Theorem~\ref{thm-3planesspl}, we obtain that
$P_0^t\cup P_1^t\cup P_2^t$ is locally 
polynomially convex at the origin.
\smallskip

\noindent{\bf Case III: $\dfrac{3}{2}\leq t^2<3$.}
\smallskip

\noindent For $\sqrt{\dfrac{3}{2}}<t<\sqrt{3}$, we see that 
\[
\det[A_1^t, A_2^t]>0,\;\;\det A_j^t<0 \;\; j=1,2,\;\;\text{and}\;\;
|\det A_j^t|=\left|\dfrac{3-t^2}{3(1-t^2)}\right|\leq 1
\]
Hence, all the conditions of Theorem~\ref{thm-planesdetlessone} are satisfied. 
Therefore, $P_0^t\cup P_1^t\cup P_2^t$ is locally polynomially convex at the origin.
\medskip

\noindent {\bf Case IV: $\dfrac{15-\sqrt{33}}{8}<t^2<2$.}
\smallskip

\noindent In this case, we have 
\[
\det[A_1^t,A_2^t]>0,\;\;\text{and}\;\; \det A^t_j<0,\;j=1,2.
\]
We now show that $\beta(A_1^t,A_2^t)<\det A_2 (Tr A_1)^2$. We compute:
\begin{align*}
TrA_1^tA_2^t-Tr A_1^t Tr A_2^t
&= \dfrac{2(t^4+2t^2-3)}{3(1-t^2)^2},\\
TrA_1^tA_2^t(TrA_1^tA_2^t-Tr A_1^t Tr A_2^t)&= -\dfrac{4(t^8-(2t^2-3)^2)}{9(1-t^2)^4},\\
\beta(A_1^t,A_2^t)
&=\dfrac{4t^2}{9(1-t^2)^4}\left[t^6-8t^4+18t^2-12\right].
\end{align*}	
Thus, we have
\begin{align*}
\det A_2^t (TrA_1^t)^2-\beta (A_1^t,A_2^t)&=	 \dfrac{4t^4(3-t^2)}{9(1-t^2)^3}-\dfrac{4t^2}{9(1-t^2)^4}(t^6-8t^4+18t^2-12)\\
&=\dfrac{4t^2}{9(1-t^2)^4}\left(t^2(3-4t^2+t^4)-t^6+8t^4-18t^2+12\right)\\
&=\dfrac{4t^2}{9(1-t^2)^4}\left(4t^4-15t^2+12\right).
\end{align*}
Since, by assumption, $\dfrac{15-\sqrt{33}}{8}<t^2<2$, we get that
$4t^4-15t^2+12<0$.
This implies $\det A_2^t(TrA_1^t)^2<\beta (A_1^t,A_2^t)$. 
Hence, using Theorem~\ref{thm-planeszw}, we conclude that 
$P_0^t\cup P_1^t\cup P_2^t$ is locally polynomially convex at the origin.
\smallskip

Therefore, combining all the cases, we conclude that $P_0^t\cup P_1^t\cup P_2^t$ is locally 
polynomially convex at the origin for every $t>\dfrac{\sqrt{15-\sqrt{33}}}{2\sqrt{2}}$.
\end{proof}

\begin{proof}[Proof of Theorem~\ref{thm-cubic-pcvx}]

$(i)$ Assume 
$t> \dfrac{\sqrt{15-\sqrt{33}}}{2\sqrt{2}}$. We will show that $S_t$ is 
locally polynomially convex at $0\in\cplx^2$. 
Given 
$S_t=\{(z,p_t(z,\ba{z})):z\in\cplx\},$
 where $p_t(z,\ba{z})=z^2\ba{z}+tz\ba{z}^2+\dfrac{t^2}{3}\ba{z}^3$. We now use 
 the proper holomorphic map $\Phi_t:\cplx^2\longrightarrow\cplx^2$ such that 
 $\Phi_t(z,w)=(z,p_t(z,w))$. In view of Result~\ref{lem-planes}, we obtain 
 \[
 \Phi_t^{-1}(S_t)=P_0^t\cup P_1^t\cup P_2^t.
 \]
 We recall that the 
 pair of matrices \eqref{eq-planemat} corresponding to Weinstock's normal form of the triple $(P_0^t,P_1^t,P_2^t)$ is 
 \[
A_1^t=\begin{pmatrix}
\dfrac{t}{\sqrt{3}(1+t)} & \dfrac{1}{1+t}\\
-\dfrac{1}{1-t} & -\dfrac{t}{\sqrt{3}(1-t)}
\end{pmatrix}
\quad\text{and}\quad
A_2^t=\begin{pmatrix}
-\dfrac{t}{\sqrt{3}(1+t)} & \dfrac{1}{1+t}\\
-\dfrac{1}{1-t} & \dfrac{t}{\sqrt{3}(1-t)}
\end{pmatrix}.
\]
By Theorem~\ref{thm-3planes}, we get that, for each $t>\dfrac{\sqrt{15-\sqrt{33}}}{2\sqrt{2}}$, 
$P_0^t\cup P_1^t\cup P_2^t$ is locally polynomially convex at the origin. 
Therefore, in view of Lemma~\ref{lem-propermap}, for $t>\dfrac{\sqrt{15-\sqrt{33}}}{2\sqrt{2}}$, the surface $S_t$ is locally polynomially 
convex at the origin. 
 \medskip
 
 $(ii)$ We use the proper map $\Phi_1$ so that 
\[
\Phi_1^{-1}(S_1)=P_0\cup P_1\cup P_2,
\]
where the planes are:
\begin{align*}
P_0^1&: \qquad w=\ba{z},\\
P_1^1&: \qquad w=-\dfrac{\sqrt{3}(\sqrt{3}-i)}{2} z-\dfrac{1-i\sqrt{3}}{2}\ba{z},\\
P_2^1&:\qquad w=-\dfrac{\sqrt{3}(\sqrt{3}+i)}{2} z-\dfrac{1+i\sqrt{3}}{2}\ba{z}.
\end{align*}
Since $P_j^1\cap P_k^1=\{(z,w)\in P_j^1: z=-\ba{z}\}$ for $0\leq j<k\leq 2$, by \cite[Lemma~1.3]{SG1}, 
the pairwise unions of $P_0, P_1$ and $P_2$ are locally polynomially convex at the origin. 
We now use the change of variables $\Psi:\cplx^2\longrightarrow\cplx^2$ defined by, 
\[
\Psi(z,w)=(z+w,i(z-w)).
\]
We denote $\wtil{P_j}:=\Psi(P_j)$, $j=0,1,2$, where
\begin{align*}
\wtil{P_0}&=\{(2x,-2y)\in\cplx^2:x,y\in\rea\},\\
\wtil{P_1}&=\{(-(1-i\sqrt{3})x, -2y+(\sqrt{3}+3i)x)\in\cplx^2: x,y\in\rea\},\\
\wtil{P_2}&=\{(-(1+i\sqrt{3})x, -2y-(\sqrt{3}-3i)x)\in\cplx^2:x,y\in\rea\}.
\end{align*} 
Let $K_j=\wtil{P_j}\cap \ba{\mathbb{B}(0;1)}$, $j=0,1,2$, and $K=K_1\cup K_2$. We consider the polynomial 
$F(z,w)=z$. Clearly, we have:
\begin{align*}
F(K_1)&\subset\rea\subset\cplx,\\
F(K_2)&\subset (1-i\sqrt{3})\rea\subset\cplx,\\
F(K_3)&\subset (1+i\sqrt{3})\rea\subset\cplx.
\end{align*}
Thus, we have $\hull{K_0}\cap \hull{K}=\{0\}$. We also obtain that 
$F^{-1}\{0\}\cap K_0=\{(0,-y)\in \cplx^2: y\in\rea\}\cap K_0$, which is polynomially convex, and 
$F^{-1}\{0\}\cap K=\{(0,-y)\in \cplx^2: y\in\rea \}\cap K$. 
Therefore, by Kallin's lemma (Result~\ref{lem-Kallin}), we conclude that 
$K_0\cup K(=K_0\cup K_1\cup K_2)$ is polynomially convex.
\end{proof}
\medskip

\begin{proof}[Proof of Theorem~\ref{thm-surface-pcvx}]  

$(i)$ 
We know that there exist 
$r>0$ such that 
\[
M_t\cap \ba{B(0;r)}=\{(z,w)\in\cplx^2: \phi(z,w)=0\},
\]
where $\phi(z,w)= p_t(z,\ba{z})+F(z,\ba{z})-w$ with $F(z)=o(|z|^3)$.
We first prove that, for $t\in(0,1)\cup (1,\infty)$, the pre-image of $M_t$, near the origin, under the proper holomorphic map $\Phi_t$, 
is a union of three pairwise transverse totally-real surfaces intersecting only at the origin.    
More precisely, we will show, shrinking $r$, if needed, that 
\[
\Phi_t^{-1}\{M_t\cap \ba{B(0;r)}\}=S^t_0\cup S^t_1\cup S^t_2,
\]
 where 
\begin{align*}
S_0^t&=\left\{(\zeta,\ba{\zeta}+f(\zeta))\in\cplx^2: \zeta\in\cplx,\; |\zeta|\leq r\right\},\\
S_1^t&=\left\{\left(\zeta,-\dfrac{3-i\sqrt{3}}{2t}\zeta- \dfrac{1-i\sqrt{3}}{2}\ba{\zeta}+g(\zeta)\right)\in\cplx^2: \zeta\in\cplx,\; |\zeta|\leq r\right\},\\
S_2^t&=\left\{\left(\zeta,-\dfrac{3+i\sqrt{3}}{2t}\zeta- \dfrac{1+i\sqrt{3}}{2}\ba{\zeta}+h(\zeta)\right)\in\cplx^2: \zeta\in\cplx,\; |\zeta|\leq r\right\},
\end{align*}
with $f(\zeta)\sim o(|\zeta|)$, $g(\zeta)\sim o(|\zeta|)$ and $h(\zeta)\sim o(|\zeta|)$ near the origin.
\smallskip

We see that
\[
\Phi_t(S_0^t)\subset M_t \impl 
\phi(\Phi(\zeta,\ba{\zeta}+f(\zeta)))=0.
\]
This implies 
\begin{align}
& \phi(\zt, p_t(\zt,\ba{\zt}+f(\zt))=0 \notag\\
\impl &
t^2(f(\zeta))^3+3t(\zeta+t\ba{\zeta}) (f(\zeta))^2+3(\zeta+t\ba{\zeta})^2 f(\zeta) - 3F(\zeta, \ba{\zt})=0.\label{eq-fcube}
\end{align}
Thus, we obtain that 
\begin{align}
\left(t f(\zeta)+(\zeta+t\ba{\zeta})\right)^3&=3tF(\zeta)+(\zeta+t\ba{\zeta})^3\notag\\
\impl t f(\zeta)+(\zeta+t\ba{\zeta})&
=(\zeta+t\ba{\zeta})\left(1+\dfrac{3tF(\zeta)}{(\zeta+t\ba{\zeta})^3}\right)^{1/3}.\label{eq-fs}
\end{align}
Since our requirement is $f(\zt)\sim o(|\zt|)$ near the origin, we have chosen the cube root of unity as $1$ in 
\eqref{eq-fs}.
Since $F(\zeta)\sim o(|\zeta|^3)$ and $(\zeta+t\ba{\zeta})^3\sim O(|\zeta|^3)$, 
\[
\left|\dfrac{3tF(\zeta)}{(\zeta+t\ba{\zeta})^3}\right|\to 0 \;\; \text{as}\;\; \zeta\to 0.
\]
Hence, in a small neighbourhood of the origin, the function $\left(1+\dfrac{3tF(\zeta)}{(\zeta+t\ba{\zeta})^3}\right)^{1/3}$ well defined and continuous.
Expanding right hand side of \eqref{eq-fs}, we get that $f$ is $\smoo^1$-smooth near the origin and $f(\zeta)\sim o(|\zeta|)$.
\smallskip

Since $\Phi_t(S_1^t)\subset M_t$, we get that 
\[
\phi\left(\Phi\left(\zeta,-\dfrac{3-i\sqrt{3}}{2t}\zeta- \dfrac{1-i\sqrt{3}}{2}\ba{\zeta}+g(\zeta)\right)\right)=0.
\]
This implies
\begin{align}
& t^2(g(\zeta))^3-3te^{2i\pi/3}(\zeta+t\ba{\zeta}) (g(\zeta))^2
+3e^{4i\pi/3}(\zeta+t\ba{\zeta}) g(\zeta)- 3F(\zeta)=0\notag\\
\impl& \left(t g(\zeta)+e^{i\pi/3}(\zeta+t\ba{\zeta})\right)^3
=(\zeta+t\ba{\zeta})^3\left(1+\dfrac{3tF(\zeta)}{(\zeta+t\ba{\zeta})^3}\right).\label{eq-gcube}
\end{align}
Thus, we have 
\[
\left|\dfrac{3tF(\zeta)}{(\zeta+t\ba{\zeta})^3}\right|\to 0 \;\; \text{as}\;\; \zeta\to 0.
\]
We now choose $e^{i\pi/3}$ as the cube root of unity to get a $\smoo^1$-smooth
$g$  in a small neighbourhood of the origin and 
\[
g(\zeta)\sim o(|\zeta|).
\]
\smallskip

  We also have $\Phi_t(S_2^t)\subset M_t$. This implies that
\[
\phi\left(\zeta,-\dfrac{3+i\sqrt{3}}{2t}\zeta- \dfrac{1+i\sqrt{3}}{2}\ba{\zeta}+h(\zeta)\right)=0,
\]
and hence, 
\begin{align}
&t^2(h(\zeta))^3+3te^{4i\pi/3}(\zeta+t\ba{\zeta}) (h(\zeta))^2
+3e^{2i\pi/3}(\zeta+t\ba{\zeta})^2 h(\zeta)- 3F(\zeta)=0.\notag\\
\impl& \left(t h(\zeta)+e^{2i\pi/3}(\zeta+t\ba{\zeta})\right)^3
=(\zeta+t\ba{\zeta})^3\left(1+\dfrac{3tF(\zeta)}{(\zeta+t\ba{\zeta})^3}\right).\label{eq-hcube}
\end{align}
In this case, chooing $e^{2i\pi/3}$ as the cube root of unity, we get from \eqref{eq-hcube} that 
$h$ is $\smoo^1$-smooth in a small neighbourhood of the origin and 
$h(\zeta)\sim o(|\zeta|).$
\smallskip

The tangent spaces of $S_0^t, S_1^t$ and $S_2^t$ at the origin are: 
\begin{align}
T_0S_0^t &=\left\{(\zeta,\ba{\zeta})\in\cplx^2: \zeta\in\cplx \right\},\notag\\
T_0S_1^t&=\left\{\left(\zeta,-\dfrac{3-i\sqrt{3}}{2t}\zeta- \dfrac{1-i\sqrt{3}}{2}\ba{\zeta}\right)\in\cplx^2: \zeta\in\cplx \right\},\notag\\
T_0S_2^t&=\left\{\left(\zeta,-\dfrac{3+i\sqrt{3}}{2t}\zeta- \dfrac{1+i\sqrt{3}}{2}\ba{\zeta}\right)\in\cplx^2: \zeta\in\cplx \right\}.\notag
\end{align}
Using a $\cplx$-linear change of coordinates, in view of \eqref{eq-planemat},
 we get Weinstock's normal form for the tangent spaces 
as 
\[
T_0(S_0^t)=\rea^2,\; T_0S_j^t=(A_j^t+i\id)\rea^2,\; j=1,2,
\]
where
\[
A_1^t=\begin{pmatrix}
\dfrac{t}{\sqrt{3}(1+t)} & \dfrac{1}{1+t}\\
-\dfrac{1}{1-t} & -\dfrac{t}{\sqrt{3}(1-t)}
\end{pmatrix}
\quad\text{and}\quad
A_2^t=\begin{pmatrix}
-\dfrac{t}{\sqrt{3}(1+t)} & \dfrac{1}{1+t}\\
-\dfrac{1}{1-t} & \dfrac{t}{\sqrt{3}(1-t)}
\end{pmatrix}.
\]
Since $t>\dfrac{\sqrt{15-\sqrt{33}}}{2\sqrt{2}}$, by Theorem~\ref{thm-3planes}, any compact subset of the union
$\rea^2\cup (A_1^t+i\id)\rea^2\cup (A_2^t+i\id)\rea^2$ 
is polynomially convex. We now divide the rest of the proof in few cases.
\smallskip

\noindent {\bf Case I: $t^2>3$.}
\smallskip

\noindent In this case, we note that
\[
\det[A_1^t,A_2^t]>0
\;\;
\text{and}\;\;
\det A_j^t>0,\;\; j=1,2.
\]
 Therefore, in view of Part $(i)$ of 
Theorem~\ref{thm-unionsurfaces}, we get that $S_0^t\cup S_1^t\cup S_2^t$ is locally polynomially 
convex at the origin.
\smallskip

\noindent{\bf Case II: $\dfrac{3}{2}< t^2<3$.}
\smallskip

\noindent For $\sqrt{\dfrac{3}{2}}<t<\sqrt{3}$, we see that 
\[
\det[A_1^t, A_2^t]>0,\;\;\det A_j^t<0 \;\; j=1,2,\;\;\text{and}\;\;
|\det A_j^t|=\left|\dfrac{3-t^2}{3(1-t^2)}\right|< 1
\]
Hence, by Part $(ii)$ of Theorem~\ref{thm-unionsurfaces}, we conclude that 
 $S_0^t\cup S_1^t\cup S_2^t$ is locally polynomially convex at the origin.
\smallskip

\noindent {\bf Case III: $t^2=3$}
\smallskip

\noindent We note that
\[
\det[A_1^t,A_2^t]>0\;\;\text{and}\;\;\det A_j^t=0,\quad j=1,2,
\]
Hence, by Theorem~\ref{thm-unionsurfacesdet0}, we obtain that
$S_0^t\cup S_1^t\cup S_2^t$ is locally 
polynomially convex at the origin.

Therefore, combining all the cases, we conclude that $S_0^t\cup S_1^t\cup S_2^t$ is locally 
polynomially convex at the origin for every $t>\sqrt{\dfrac{3}{2}}$.
 Thus, in view of Result~\ref{lem-propermap}, 
we conclude that $M_t$ is locally polynomially convex at $0\in\cplx^2$ if $t>\sqrt{\dfrac{3}{2}}$.
 \end{proof}
 

\section{Proofs of  Theorem~\ref{thm-planes-hull}, Theorem~\ref{thm-cubic-hull} and Theorem~\ref{thm-surface-hull}}\label{sec-surface-hull}

In this section, we first prove Theorem~\ref{thm-cubic-hull} and we apply it to prove Theorem~\ref{thm-planes-hull}. Finally, we give a proof of Theorem~\ref{thm-surface-hull}. We begin with a lemma in one variable complex analysis.
\begin{lemma}\label{lem-preimage}
Let $g_t(z)=z^2+tz^4+\dfrac{t^2}{3}z^6$, $0\leq t<1$. For each point $a\in \bdy \mathbb{D}$, $g_t$ 
has only two pre-image of $g_t(a)$ on $\bdy\mathbb{D}$ if and only if $0\leq t<\sqrt{3}/2$.
\end{lemma}

\begin{proof}[Proof of Lemma~\ref{lem-preimage}] 
Consider a point $a\in \bdy \mathbb{D}$. The equation $g_t(z)=g_t(a)$ gives us
\[
z^2+tz^4+\dfrac{t^2}{3}z^6 = a^2+ta^4+\dfrac{t^2}{3}a^6.
\]
If $t=0$, then there are only two roots $a,-a$ of the above equation. Therefore, we assume $t\neq 0$.
The above equation gives us
\begin{align*}
& \impl (z+a)(z-a)+t(z^4-a^4)+\dfrac{t^2}{3}(z^6-a^6)=0\\
& \impl(z-a)(z+a)(1+t(z^2+a^2)+\dfrac{t^2}{3}(z^4+a^2z^2+a^4))=0
\end{align*}
Hence, the set $g_t^{-1}\{g_t(a)\}$ consists of $a, -a$, and the roots of the following quartic equation: 
\[
\dfrac{t^2}{3}z^4+t(1+\dfrac{ta^2}{3})z^2+1+ta^2+\dfrac{t^2}{3}a^4=0.
\]
To find the roots of the above quartic we make a change of variable $w=z^2$, where 
$w$ is the new variable. The quartic now reduces to a quadratic equation: 
\[
\dfrac{t^2}{3}w^2+t(1+\dfrac{ta^2}{3})w+1+ta^2+\dfrac{t^2}{3}a^4=0.
\]
The roots are 
\[
w= -\dfrac{3}{2t}(1+\dfrac{ta^2}{3})\pm \dfrac{\sqrt{3} i}{2t}(1+ta^2) 
\]
Since $a\in\bdy \mathbb{D}$, $a=e^{i\psi}$ for some $\psi\in [0,2\pi]$. 
The roots of the above quadratic are: 
\begin{align*}
\lambda_1 &:= -\dfrac{3}{2t}(1+\dfrac{t}{3}e^{2i\psi})+ \dfrac{\sqrt{3} i}{2t}(1+te^{2i\psi}).\\
\lambda_2 &:= -\dfrac{3}{2t}\left(1+\dfrac{t}{3}e^{2i\psi}\right) -\dfrac{\sqrt{3} i}{2t}\left(1+te^{2i\psi}\right).
\end{align*}
The real and the imaginary parts of $\lambda_1$ are:
\begin{align*}
\rl{\lambda_1}&=-\dfrac{3}{2t}-\cos(2\psi-\pi/3).\\
\imag{\lambda_1}&= \dfrac{\sqrt{3}}{2t}-\sin(2\psi-\pi/3).
\end{align*}
Hence, 
\[
|\lambda_1|^2= \dfrac{3}{t^2}+\dfrac{3}{t}\cos(2\psi-\pi/3)-\dfrac{\sqrt{3}}{t}\sin(2\psi-\pi/3)+1.
\]
Thus, we have 
\[
|\lambda_1|=1\iff \dfrac{3}{t}+3\cos(2\psi-\pi/3)-\sqrt{3}\sin(2\psi-\pi/3) =0.
\]
We now consider a function $g:[0,2\pi]\to\rea$ defined by
\[
g(x):=3\cos x-\sqrt{3}\sin x.
\]
We need to find the minimum value of $g$ in $[0,2\pi]$. We compute 
$g'(x)=-3\sin x-\sqrt{3}\sin x$ and $g'(x)=0$ implies $x=\pi-\pi/6$ and $x=2\pi-\pi/6$. The function 
$g$ attains its minimum at $x=\pi-\pi/3$, and the minimum value is $-2\sqrt{3}$. Hence, 
$|\lambda|=1$ if and only if 
\[
3/t\leq2\sqrt{3} \iff t\geq \dfrac{\sqrt{3}}{2}.
\]
 By similar argument, we obtain that
 $|\lambda_2|^2=1$ if and only if 
 \[
 \dfrac{3}{t}+3\cos(2\psi+\pi/3)+\sqrt{3}\sin(2\psi+\pi/3)=0.
 \]
 That is, $|\lambda_2|=1$ if and only if $t\geq \dfrac{\sqrt{3}}{2}$. Hence, 
 the number of pre-image of $g_t(a)$ under $g_t$ is $2$ if and only if $0\leq t<\dfrac{\sqrt{3}}{2}$. 
 \end{proof}

We now present our proof of Theorem~\ref{thm-cubic-hull}.
\begin{proof}[Proof of Theorem~\ref{thm-cubic-hull}]
 We first 
  compute the Maslov-type index for $S_t$ at $0\in\cplx^2$. 
 Given $p_t(z,\ba{z})= z^2\ba{z}+t z\ba{z}^2+\dfrac{t^2}{3} \ba{z}^3$. 
 Putting $z=x+iy$, we get 
 \[
 \dfrac{\bdy p_t}{\bdy \ba{z}}(z,\ba{z})=(1+t)^2x^2-(1-t)^2 y^2+2i(1-t^2)xy.
 \]
 For $t\neq 1$, clearly, $\left(\partial p_t/\bdy \ba{z}\right)^{-1}\{0\}=\{0\}$.
 In this case, we have 
  \[
  q_t(z)=\dfrac{\bdy p_t}{\bdy \ba{z}}(z,1)=(z+t)^2.
  \]
 Therefore, using Result~\ref{res-maslov}, we obtain that the  Maslov-type index:
 \begin{equation}\label{eq-masindex}
 Ind_{S_t}(0)=\begin{cases}
                                  2 ,\; &\text{if}\; t<1\\
                                  -2, \;&\text{if}\; t>1.
                                  \end{cases}
                                  \end{equation}
\medskip

Next,
 we show that the function  $\phi_t(z,\ba{z}):=\rl\left(\dfrac{p_t(z,\ba{z})}{z}\right)$ is subharmonic in $\cplx\setminus\{0\}$ for 
$t<1$. We have
\[
\dfrac{p_t(z,\ba{z})}{z}=|z|^2+t\ba{z}^2+\dfrac{t^2}{3} \dfrac{\ba{z}^3}{z}.
\]
Hence, 
\[
\phi_t(z,\ba{z})= z\ba{z}+\dfrac{t}{2}(z^2+\ba{z}^2) + 
\dfrac{t^2}{6}\left(\dfrac{z^4+\ba{z}^4}{|z|^2}\right).
\]
\begin{align}
\dfrac{\bdy^2\phi_t}{\bdy z\bdy\ba{z}}(z,\ba{z})&=1-\dfrac{t^2(z^4+\ba{z}^4)}{|z|^4}\notag\\
&= 1-t^2\rl\left(\dfrac{z^2}{\ba{z}^2}\right) \label{eq-lapphi}.
\end{align}
Hence, in view of \eqref{eq-lapphi}, we get that $\phi_t$ is subharmonic in $\cplx\setminus\{0\}$ for all 
$t\in (0,1)$.
\smallskip

 $(i)$ Assume that $t<1$. In view of \eqref{eq-masindex}, the Maslov-type index 
$Ind_{S_t}(0)=2$ and $\rl\left(\dfrac{p_t(z,\ba{z})}{z}\right)$ is subharmonic in $\cplx\setminus\{0\}$
Therefore, applying part $(i)$ Result~\ref{res-wiegerinck},
we conclude that $S_t$ is not locally polynomially convex at the origin for $t<1$.   
\smallskip

 $(ii)$ By Part $(i)$, we already know that the local hull of $S_t$ is nontrivial. We will use 
Part $(ii)$ of Result~\ref{res-wiegerinck}. For that,
we consider the curve defined on the unit circle:
\begin{align}
\mathscr{C}_t(z)&=\dfrac{p_t(z,\ba{z})}{z^3} \notag\\
&= \ba{z}^2+t\ba{z}^4+\dfrac{t^2}{3}\ba{z}^6,\;\; |z|=1.\label{eq-curvec}
\end{align}
Since $t\in\rea$, we see that $\ba{\mathscr{C}_t(z)}=\mathscr{C}_t(\ba{z})$. 
We now consider 
the polynomial
\begin{equation}\label{eq-auxcurve}
g_t(z):=z^2+tz^4+\dfrac{t^2}{3}z^6.
\end{equation}
By Lemma~\ref{lem-preimage}, we obtain, under the condition $0\leq t<\sqrt{3}/2$, the 
pre-image of $g_t(a)$ under $g_t$ has exactly two points and  
the points are $a$ and $-a$ for any $a\in\bdy\disc$. Hence, if for any two distinct points 
$z_1, z_2\in\bdy\disc$, $\mathscr{C}_t(z_1)=\mathscr{C}_t(z_2)$, then $z_1=-z_2$. Therefore, 
$z_1, z_2$ divides the unit circle in two segments of lengh $\pi>\pi/2$. 
Hence, part $(ii)$ of Result~\ref{res-wiegerinck} applies in this situation. Therefore, for $0\leq t<\sqrt{3}/2$, the polynomial hull of $S_t\cap \ba{B(0;\dl)}$ contains a ball centred at the origin and with positive radius.
\smallskip

$(iii)$ We assume $\sqrt{3}/2\leq t<1$. In this case, using  Lemma~\ref{lem-preimage}, we obtain  
that the set $g_t^{-1}\{g_t(a)\}$ has more than two points in the unit circle. The proof of Lemma~\ref{lem-preimage} shows that it has exactly four points in the set $g_t^{-1}\{g_t(a)\}\cap \bdy\disc$ if $\sqrt{3}/2\leq t<1$, 
and the set is not of the form $\{a,-a,ia,-ia\}$. Hence,
the curve $\mathscr{C}_t(z)$ defined 
as in \eqref{eq-curvec} does not satisfy Property ${\sf (**)}$ in Result~\ref{res-wiegerinck}. Hence, by Part $(iii)$ 
of Result~\ref{res-wiegerinck}, we obtain, at least, a one parameter family of analytic discs with boundary lying in 
$S_t$ passing through the origin.
\end{proof}
 
 \begin{proof}[Proof of Theorem~\ref{thm-planes-hull}] 
 Let us fix a $\dl>0$ and let $K=(P_0^t\cup P_1^t\cup P_2^t)\cap \Phi_t^{-1}(B(0;\dl))$.  We note that 
 $\Phi_t^{-1}(\Phi_t(K))=K$.
 Clearly, $\Phi_t(K)\subset S_t\cap B(0;\dl)$ is a compact subset containing a neighbourhood the origin in $S_t$. 
 By Part $(i)$ of Theorem~\ref{thm-cubic-hull}, for $t<1$, $\Phi_t(K)$ is not polynomially convex. Therefore, 
 $K$ is not polynomially convex. Since $\dl>0$ is chosen arbitrarily, $P_0^t\cup P_1^t\cup P_2^t$ is not locally 
 polynomially convex at the origin.
 
 $(i)$ We consider $K$ as before. By Result~\ref{lem-propermap}, 
 we get that $\Phi_t(\hull{K})$ is polynomially convex. Since $\Phi_t(K)$ contains a neighbourhood of the origin
 in the relative topology of $S_t$, using Part $(ii)$ of Theorem~\ref{thm-cubic-hull}, we obtain that
  $\hull{\Phi_t(K)}$ contains a ball $B(0;r)$ for some $r>0$. Since $\Phi_t(\hull{K})$ is polynomially 
 convex and contains $\Phi_t(K)$. Hence, 
 \begin{equation}\label{eq-prop}
 \hull{\Phi_t(K)}\subset \Phi_t(\hull{K}).
 \end{equation}
 Therefore, $\Phi_t^{-1}(B(0;r))\subset \hull{K}$, which contains a neighbourhood of the origin in $\cplx^2$.
 
 $(ii)$ Using same 
 $K$ as in Part~$(i)$, we see that, for $\dfrac{\sqrt{3}}{2}\leq t<1$, $\hull{\Phi_t(K)}$ contains at least a one parameter family of analytic discs passing through the origin. We also have $\Phi_t^{-1}\{0\}=\{0\}$. 
 Therefore, in view of \eqref{eq-prop}, the pre-images of the analytic discs lie in $\hull{K}$. Boundary of the each of the pre-images of analytic discs passes through the origin. Pre-image of a disc is an analytic variety with boundary in $K$.
 \end{proof}
 
 \begin{proof}[Proof of Theorem~\ref{thm-surface-hull}]
  $(i)$ We assume $t<1$. 
  In view of Lemma~\ref{lem-maslovwinding}, we get that 
 \[
 Ind_{S_t}(0)=Ind_{M_t}(0).
 \]
 In the proof of Part $(i)$ of Theorem~\ref{thm-cubic-hull}, 
 the Maslov-type index for $S_t$ is computed.
 Hence, for $t\in (0,1)$, the Maslov-type index
 \[
 Ind_{M_t}(0)=2.
 \]
 We also know that $\rl{\dfrac{p_t(z,\ba{z})}{z}}$ is strictly subharmonic in $\cplx\setminus\{0\}$. Since a small
 $\smoo^2$-perturbation of strictly subharmonic function is again strictly subharmonic and 
 $F(z,\ba{z})\sim o(|z|^3)$ near the origin, we obtain that 
 $\rl{\dfrac{p_t(z,\ba{z})+F(z,\ba{z})}{z}}$ is strictly subharmonic in a small deleted neighbourhood of the origin.
 Therefore, by Result~\ref{res-wiegerinck1}, we conclude that there is an analytic disc with boundary in $M_t$. Hence, $M_t$ is not locally polynomially convex at the origin.
\medskip

$(ii)$
We will use Result~\ref{res-wiegnonhomo} to prove this.
We know that
\[
 Ind_{M_t}(0)=2 \quad\forall 0<t<1.
 \]
and $\rl{\dfrac{p_t(z,\ba{z})}{z}}$ is subharmonic but nowhere harmonic in $\cplx\setminus\{0\}$. 
Since $0<t<\sqrt{3}/2$, by Lemma~\ref{lem-preimage}, $S_t$ satisfies Property ${\sf(**)}$ (see Part $(ii)$ of Theorem~\ref{thm-cubic-hull} for computations). 
Hence all the conditions of Result~\ref{res-wiegnonhomo} are satisfied. Thus, for every $r>0$, the polynomial hull of $M_t\cap \ba{B(0;r)}$ contains a ball with centred at the origin and of positive radius.  
\medskip

$(iii)$ Assume $\dfrac{\sqrt{3}}{2}<t<1$.
We know that the 
Maslov-type index of $S_t$, the graph of lowest order polynomial of the local graphing function 
near the origin, is $2$. 
By Part $(i)$, we already know that there is an analytic disc with boundary on $M_t$. 
Actually, the proof of Result~\ref{res-wiegerinck1} in \cite{Wieg}
gives a one parameter family.
If the disc encloses the CR singularity at the origin, then, by using 
Result~\ref{res-andisc}, we obtain a three parameter family of analytic discs that will fill a neighbourhood of the origin. Otherwise, the boundaries of the discs must pass through the origin. 

\end{proof}

\section{Questions}\label{sec-questions}
In this section we mention few questions whose answers will make the picture clearer. The questions seem quite interesting (at least to the author) and classical. We 
begin with a conjecture.
\begin{conjecture}
 The surface $M_t$ is locally polynomially convex if $t>1$.
\end{conjecture} 
 \noindent Theorem~\ref{thm-surface-pcvx} proves above 
 for $t> \sqrt{\dfrac{3}{2}}=1.2247..$. The only part remains open is for 
 $1<t\leq \sqrt{\dfrac{3}{2}}$. 

\begin{question}
Does there exist an example of a surface $M_t$ for $\sqrt{3}/2\leq t<1$ such that the local polynomial hull 
at the origin contains a nonempty open ball centred at the origin? Or does the discs attached to it always have  boundaries passing
through the origin?
\end{question}
\noindent Wiegerinck's result just gives a one parameter family on analytic discs. We do not know 
whether an if and only if condition for obtaining a three parameter family of analytic discs, analogous 
to Result~\ref{res-wiegerinck}, is possible when the higher order terms are present. In any case, it will be interesting to know the local hull of $M_t$ for $\sqrt{3}/2\leq t<1$.

\begin{question}
	What can one say about the fine structure of the local hull of $S_t$? Will it contain an interior point if $\dfrac{\sqrt{3}}{2}\leq t<1$? 
\end{question}

\begin{question}
	What can one say about the local polynomial convexity of $M_t$ for $t=1$, i.e., at the parabolic CR 
	singularity of higher order?
\end{question}

\begin{question} 
What can one say about the local hulls and local polynomial convexity about the surfaces with cubic lowest order homogeneous term that do not belong to this family? Does there exist any other one parameter family exhibiting Bishop-type dichotomy?
\end{question}

\begin{question}
Can one achieve  analogous results for surfaces with CR singularity of higher order, i.e. when $M_t$, locally at the origin, 
is of the form 
$\{(z,w)\in\cplx^2: w=(z+t\ba{z})^k+o(|z|^k)\},\; k>3$? 
\end{question}
\noindent In this case, one can not apply Result~\ref{res-wiegerinck1} to get nontrivial hull for $t<1$ as the subharmonicity condition that appears in Results~\ref{res-wiegerinck1}, \ref{res-wiegerinck}, and 
\ref{res-wiegnonhomo} is not automatic, which was the case for $k=3$.
\medskip

\noindent {\bf Acknowledgements.} I would like to thank Gautam Bharali for showing the proof of 
Lemma~\ref{lem-planesunique}. I would like to thank Buddhananda Banerjee for helping me with the software R 
to visualize the curve $C(z)$, which
confirmed my guess of the number $\sqrt{3}/2$.
This work is partially supported by 
Mathematical Research Impact Centric Support (MATRICS) grant, File No: MTR/2017/000974, by the Science and Engineering Research Board (SERB), Department of Science \& Technology (DST), Government of India.


\begin{thebibliography}{99}
 \bibitem{AW}
 L. Arosio and E. F. Wold,
 {\em Totally real embeddings with prescribed polynomial hull},
 Indiana Univ. Math. J. {\bf 68} (2019), no. 2, 629-640.
 
 \bibitem{BedGav}
 E. Bedford and B. Gaveau,
 {\em Envelopes of holomorphy of certain $2$-spheres in $\cplx^2$},
 Amer. J. Math. {\bf 105} (1983), 975-1009.
 
 \bibitem{BedKling} 
 E. Bedford and W. Klingenberg Jr., 
 {\em On the envelope of holomorphy of a $2$-sphere in $\cplx^2$},
 J. Amer. Math. Soc. {\bf 4} (1991), 623-646.
 
 \bibitem{EB}
 E. Bishop,
 {\em Differentiable manifolds in complex Euclidean space},
 Duke Math. J. {\bf 32} (1965), 1-21.

\bibitem{GB1}
G. Bharali,
{\em Surfaces with degenerate CR singularities that are locally polynomially convex},
Michigan Math. J.  {\bf 53}  (2005),  no. 2, 429-445.

\bibitem{GB2}
G. Bharali,
{\em Polynomial approximation, local polynomial convexity, and degenerate CR singularities},
J. Funct. Anal. {\bf 236} (2006), no. 1, 351-368.

\bibitem{GB3}
G. Bharali,
{\em The local polynomial hull near a degenerate CR singularity: Bishop discs revisited},
Math. Z. {\bf 271} (2012), no. 3-4, 1043-1063. 


 \bibitem{FF01}
 F. Forstneri\v{c},
 {\em Stability of polynomial convexity of totally real sets}
  Proc. Amer. Math. Soc. {\bf 96} (1986), no. 3, 489-494.
 
\bibitem{FF1}
F. Forstneri\v{c},
{\em Analytic disks with boundaries in a maximal real submanifold of $\cplx^2$},
Ann. Inst. Fourier (Grenoble) {\bf 37} (1987), no. 1, 1-44. 


\bibitem{FS}
 F. Forstneri\v{c} and E. L. Stout,
{\em A new class of polynomially convex sets},
Ark. Mat. {\bf 29} (1991), no. 1, 51-62.

\bibitem{SG1}
S. Gorai,
{Local polynomial convexity of union of two totally-real surfaces at their intersection}, 
Manuscripta Math. {\bf 135} (2011), no. 1-2, 43-62.

\bibitem{SG2}
S. Gorai,
{\em On the polynomial convexity of the union of three totally-real planes in $\cplx^2$}, 
Int. Math. Res. Not. IMRN. {\bf 2013} (2013), no. 21, 4985-5001. 


\bibitem{GH1}
 G. A. Harris, 
 {\em Lowest order invariants for real-analytic surfaces in $\cplx^2$}, Trans. Amer. Math. Soc. {\bf 288} (1985), no. 1, 413-422.
 
   
\bibitem{GH3}
 G. A. Harris, 
 {\em Degenerate surfaces in $\cplx^2$},
  Several complex variables and complex geometry, Part 3 (Santa Cruz, CA, 1989), 179-189, Proc. Sympos. Pure Math., {\bf 52}, Amer. Math. Soc., Providence, RI, 1991.


\bibitem{HW}
L. H\"{o}rmander and J. Wermer, 
{\em Uniform approximation on compact sets in $\cplx^n$}, 
Math. Scand. {\bf 23} (1968), 5-21.

\bibitem{HuntWell}
L. R. Hunt and R. O. Wells Jr., 
{\em The envelope of holomorphy of a two-manifold in $\cplx^2$},
Rice Univ. Studies {\bf 56} (1970), no. 2, 51-62.

\bibitem{J1}
B. J\"{o}ricke,
{\em Local polynomial hulls of discs near isolated parabolic points},
Indiana Univ. Math. J. {\bf 46} (1997), no. 3, 789-826.

\bibitem{K}
E. Kallin,
{\em Fat polynomially convex sets},
Function Algebras, Proc. Internat. Sympos. on Function Algebras, Tulane University, 1965, Scott Foresman,
Chicago, IL, 1966, 149-152.

\bibitem{KanSten}
H. Kang and B. Stens\o nes,
{\em Envelope of holomorphy of certain two-manifolds in $\cplx^2$},
Internat. J. Math. {\bf 22} (2011), no. 5, 713-730. 


\bibitem{KenWeb}
C. E. Kenig and S. M. Webster, 
{\em The local hull of holomorphy of a surface in the space of two complex variables}, 
Invent. Math. {\bf 67} (1982), no. 1, 1-21. 

\bibitem{LW}
E. L\o w and E. F. Wold,
{\em Polynomial convexity and totally real manifolds}, 
Complex Var. Elliptic Eq. {\bf 54} (2009), no. 3, 265-281.

\bibitem{MosWeb}
J. K. Moser and S. M. Webster, 
{\em Normal forms for real surfaces in $\cplx^2$ near complex tangents and hyperbolic surface transformations},  
Acta Math. {\em 150} (1983), no. 3-4, 255-296. 


\bibitem{OPW}
A. G. O'Farrell, Preskenis and Walsh,
{\em Holomorphic approximation in Lipschitz norms},
 Proceedings of the conference on Banach algebras and several complex variables (New Haven, Conn., 1983), 187-194, Contemp. Math., {\bf 32}, Amer. Math. Soc., Providence, RI, 1984. 




\bibitem{dP2}
P. J. de Paepe,
{\em Eva Kallin's lemma on polynomial convexity},
Bull. London Math. Soc. {\bf 33} (2001), no. 1, 1-10.

\bibitem{ShafSukh}
R. Shafikov and A. Sukhov, 
{\em Polynomially convex hulls of singular real manifolds},
Trans. Amer. Math. Soc. {\bf 368} (2016), no. 4, 2469-2496.

\bibitem{stout}
E. L. Stout,
{\em Polynomial Convexity},
Birkh{\"a}user, Boston, 2007.



\bibitem{T2}
P. J. Thomas,
{\em Enveloppes polynomiales d'unions de plans r\'{e}els dans $\cplx^n$},
Ann. Inst. Fourier (Grenoble) {\bf 40} (1990), no. 2, 371-390.

\bibitem{T1}
P. J. Thomas,
{\em Unions minimales de $n$-plans r\'{e}els d'enveloppe \'{e}gale \`{a} $\cplx^n$},
Several Complex Variables and Complex Geometry, Part 1 (Santa Cruz, CA, 1989), 233-244,
Proc. Sympos. Pure Math., {\bf 52}, Part 1, Amer. Math. Soc., Providence, RI, 1991.

\bibitem{Wk}
B. M. Weinstock,
{\em On the polynomial convexity of the union of two totally-real subspaces of $\cplx ^n$ },
Math. Ann. {\bf 282} (1988), no. 1, 131-138.

\bibitem{Wieg}
J. Wiegerinck,
{\em Local polynomially convex hulls at degenerated CR singularities of surfaces in $\cplx^2$},
Indiana Univ. Math. J. {\bf 44} (1995), no. 3, 897-915.

\bibitem{Wermer1}
J. Wermer, 
{\em Approximations on a disc},
Math. Ann. {\bf 155} (1964), 331-333.

\end{thebibliography}
\end{document}